\documentclass[11pt]{article}

\usepackage[a4paper,margin=1in]{geometry}

\usepackage[utf8]{inputenc}

\usepackage[ngerman,british]{babel}

\usepackage[T1]{fontenc}
\usepackage{mathpazo,tgcursor,tgpagella}
\usepackage{textcomp}

\usepackage[bookmarksnumbered,bookmarksopen,unicode]{hyperref}
\hypersetup{
	pdftitle={Set characterizations and convex extensions for geometric convex-hull proofs},
	pdfauthor={Andreas Bärmann,Oskar Schneider},
}

\usepackage{amsthm}

\usepackage{subcaption}
\usepackage{authblk}

\usepackage{algorithm}
\usepackage{algorithmicx}
\algnewcommand{\IIf}[1]{\State\algorithmicif\ #1\ \algorithmicthen}
\algnewcommand{\EndIIf}{\unskip\ \algorithmicend\ \algorithmicif}
\usepackage{mathdots}

\usepackage{amsmath,amssymb}
\usepackage{eurosym}
\usepackage{color}
\usepackage{eqparbox, array}
\usepackage{algpseudocode}

\algrenewcommand\algorithmicrequire{\textbf{Input:}}
\algrenewcommand\algorithmicensure{\textbf{Output:}}

\usepackage{xspace}

\usepackage{float}
\usepackage{booktabs}

\usepackage{graphicx}
\usepackage{color}
\usepackage{epstopdf}
\DeclareGraphicsExtensions{.pdf,.eps,.png}
\usepackage{stackengine}
\usepackage{tikz}
\usetikzlibrary{arrows.meta,positioning}
\usetikzlibrary{shapes.misc}
\usetikzlibrary{shapes,arrows,chains}
\tikzstyle{line} = [draw, -latex']

\usepackage{chngcntr}
\counterwithout{equation}{section}

\usepackage{paralist}
\usepackage{enumitem}

\usepackage{bbm}

\usepackage{calc}

\usepackage{todonotes} 

\presetkeys{todonotes}{inline}{}

\usepackage{numprint}
\npthousandsep{\,}

\usetikzlibrary{patterns}
\usepackage{mathdots}
\definecolor{lightmauve}{rgb}{0.86, 0.82, 1.0}

\usepackage{subcaption}

\newtheorem{theorem}{Theorem}[section]
\newtheorem{corollary}[theorem]{Corollary}
\newtheorem{definition}[theorem]{Definition}

\newtheorem{lemma}[theorem]{Lemma}

\DeclareSymbolFont{symbolsC}{U}{pxsyc}{m}{n}
\DeclareMathSymbol{\coloneqq}{\mathrel}{symbolsC}{"42}

\newcommand{\ZE}{ \bar{Z} }
\newcommand{\LLE}[1]{ \bar{\LL}^{#1} }
\newcommand{\muE}[1]{ \bar{\mu}^{#1} }
\newcommand{\phiE}[1]{ \bar{\phi}^{#1} }
\newcommand{\varphiE}[1]{\bar{\varphi}^{#1}}
\newcommand{\LE}[1]{\bar{L}^{#1}}

\newcommand{\outedges}[1]{
	\delta^{-}(#1)
}

\newcommand{\inedges}[1]{
	\delta^{+}(#1)
}

\newcommand{\y}{\bar{y}}

\renewcommand{\S}{S}

\newcommand{\F}{\{0, 1\}}
\newcommand{\FR}{[0, 1]}
\newcommand{\N}{\mathbbm{N}}
\newcommand{\Z}{\mathbbm{Z}}

\newcommand{\R}{\mathbbm{R}}

\providecommand{\myfloor}[1]{\left \lfloor #1 \right \rfloor }

\newcommand{\SSet}[2]{
		\left\{\, #1 \ \middle\vert \ #2 \, \right\}
	}

\newcommand{\func}[3]{
	#1 \colon #2  \to #3
}

\DeclareMathOperator{\cav}{cav}
\DeclareMathOperator{\vex}{vex}
\DeclareMathOperator{\LB}{LB}
\DeclareMathOperator{\UB}{UB}

\DeclareMathOperator{\conv}{conv}
\DeclareMathOperator{\cone}{cone}

\newcommand{\card}[1]{\lvert #1 \rvert}

\newcommand{\abs}[1]{\lvert #1 \rvert}
\newcommand{\norm}[1]{\lVert #1 \rVert}

\newcommand{\set}[1]{\{#1\}}
\newcommand{\lrset}[1]{\left\{#1\right\}}

\newcommand{\qm}[1]{``#1''}

\newcommand{\ie}{i.e.\ }
\newcommand{\Wlog}{w.l.o.g.\ }

\newcommand{\WLOGc}{W.l.o.g.}
\newcommand{\cf}{cf.\ }
\newcommand{\eg}{e.g.\ }

\newcommand{\mycomment}[1]{}

\newcommand{\st}{\mathrm{s.t.}}

\newlength\mysinglespace
\newlength\objspace
\newlength\conspace
\newlength\cconspace
\newenvironment{constr}[1]{\begin{array}[t]{#1}}{\end{array}} 

\newenvironment{opt*}[3]{\begin{equation*}\begin{array}{rl}#1 & #2 \\[\objspace] \st & \begin{constr}{#3}}{\end{constr}\end{array}\end{equation*}\\[0pt]}

\newenvironment{eqns*}[1]{\begin{equation*}\begin{array}[t]{#1}}{\end{array}\end{equation*}\\[0pt]}

\newcommand{\GG}{\mathcal{G}}
\newcommand{\VV}{\mathcal{V}}
\newcommand{\EE}{\mathcal{E}}

\newcommand{\Gij}{G_{ij}}

\newcommand{\Eij}{E_{ij}}

\newcommand{\FF}{\mathcal{F}}
\newcommand{\LL}{\mathcal{L}}
\newcommand{\RR}{\mathcal{R}}
\newcommand{\mP}{\mathcal{P}}
\newcommand{\QQ}{\mathcal{Q}}

\newcommand{\supplement}{online supplement \cite{supplement}}
\newif\ifproof
\prooftrue 
\definecolor{lightmauve}{rgb}{0.86, 0.82, 1.0}

\DeclareSymbolFont{symbolsC}{U}{pxsyc}{m}{n}
\DeclareMathSymbol{\coloneqq}{\mathrel}{symbolsC}{"42}
\usepackage{cleveref}
\crefname{equation}{}{}
\renewcommand{\cref}{\Cref}

\title{Set characterizations and convex extensions\\ for geometric convex-hull proofs}

\author[1]{Andreas~Bärmann}
\author[2]{Oskar Schneider\vspace{2\baselineskip}}
\affil[1]{
\url{Andreas.Baermann@math.uni-erlangen.de}

Lehrstuhl für Wirtschaftsmathematik,

Department Mathematik,

Friedrich-Alexander-Universität Erlangen-Nürnberg,

Cauerstraße 11, 91058 Erlangen, Germany
\vspace{\baselineskip}
}
\affil[2]{
\url{Oskar.Schneider@fau.de}

Gruppe Optimization 

Fraunhofer Arbeitsgruppe für Supply-Chain Services SCS,

Fraunhofer Institut für Integrierte Schaltungen IIS,

Nordostpark 93, 90411 Nürnberg, Germany
}

\date{First Draft Online: 22 January 2021}

\begin{document}

\maketitle

\begin{abstract}
In the present work, we consider Zuckerberg's method for \mbox{geometric} convex-hull proofs
introduced in [Geometric proofs for convex hull defining formulations, Operations Research Letters \textbf{44}(5), 625–629 (2016)]. 
It has only been scarcely adopted in the literature so far,
despite the great flexibility in designing algorithmic proofs
for the completeness of polyhedral descriptions that it offers.
We suspect that this is partly due to the rather heavy algebraic framework
its original statement entails.
This is why we present a much more lightweight and accessible approach
to Zuckerberg's proof technique, building on ideas from
[Extended formulations for convex
hulls of some bilinear functions, Discrete Optimization \textbf{36}, 100569 (2020)]. 
We introduce the concept of set characterizations
to replace the set-theoretic expressions needed in the original version
and to facilitate the construction of algorithmic proof schemes.
Along with this, we develop several different strategies to conduct
Zuckerberg-type convex-hull proofs.
Very importantly, we also show that our concept
allows for a significant extension of Zuckerberg's proof technique.
While the original method was only applicable to $0$/$1$-polytopes,
our extended framework allows to treat arbitrary polyhedra
and even general convex sets.
We demonstrate this increase in expressive power
by characterizing the convex hull of Boolean and bilinear functions over polytopal domains.
All results are illustrated with indicative examples
to underline the practical usefulness and wide applicability of our framework.
\\\\
\textbf{Keywords:} Convex-Hull Proofs, Zuckerberg's Method, Proof-by-Picture Method, Set Characterizations, Integer Polytopes
\\\\
\textbf{Mathematics Subject Classification:}
	90C57 - 
	52B05 - 
	90C10 - 
	90C27 - 
	90C25 
\end{abstract}

\section{Introduction}

Studying polyhedral structures lies at the heart of mixed-integer programming.
It is well-known to anyone in the field
that a good understanding of the facial structure
of a given integer linear optimization problem
both informs theory and practical algorithm development
in a very beneficial way.
This include tight problem relaxations, extended formulations,
cutting plane algorithms, only to name a few.
A more or less complete understanding of a polyhedral feasible set
can be claimed if one accomplishes a so-called convex-hull proof,
\ie a proof that a given inequality description is sufficient
to describe all of its facets.
The book \cite{PochetWolsey2006} includes a popular list
of possible approaches to obtain such a proof.
They include total unimodularity, TDI-ness, projection
or a direct proof that all vertices are integral,
among a couple of others.
For all of these approaches, there are numerous examples
where they have been used successfully,
and typically each of these methods
works especially well for particular types of problems
(such as total unimodularity for network-type problems
or TDI-ness for balanced matrices)

A relatively new technique for convex-hull proofs
has been given in \cite{zuckerberg2016geometric} by Zuckerberg,
with precursors in \cite{bienstock2004subset,zuckerberg2004set,LovaszSchrijver1991}.
It is a geometric approach based on subset algebra.
The core of Zuckerberg's method is a novel type of criterion
for showing that any given point within a given polytope (in H-description)
lies within the polytope for which a convex-hull description is to be proved.
It works by constructing an implicit (rather than an explicit), set-theoretic representation
of this point as a convex combination of the vertices of the latter.
The actual proof takes the form of an algorithm
which constructs such a set-theoretic representation.
In a second step, it is even possible to obtain the convex combination in explicit form,
in contrast to most other known proof techniques for convex-hull results.
Zuckerberg himself refers to his method either as \emph{geometric convex-hull proofs}
or as \emph{the proof-by-picture method},
because this algorithm and its result can be visualized in a diagram
incorporating all necessary information.
For the sake of simplicity, we will use the name \emph{Zuckerberg's method}
throughout to refer to this technique as well as our extensions of it.

Although the examples outlined in \cite{zuckerberg2016geometric}
already convey the impression of a very powerful proof technique,
it has only been scarcely adopted in the literature so far.
We assume that this is due to the rather heavy algebraic framework
that has been used to derive and state the method.
Zuckerberg has stated his method in terms of abstract measure spaces
over which set-theoretic expressions have to be derived.
In the recent work \cite{gupte2020extended},
the authors give a significantly simplified version of his approach
by passing over to a concrete measure space:
a real interval equipped with the Lebesgue measure.
Their actual aim in this article are proofs on the facial structure
of the graphs of bilinear functions.
However, they also give a short introduction to his proof technique
and find a way to state it mostly without using set-algebraic terms.
They proceed by showing it to be a very suitable means of proving their convex-hull results.
In \cite{harris2020convex}, the authors continue
the work of \cite{gupte2020extended}
and give further convex-hull results on special graph classes.
The authors of \cite{BMS2020} have adopted their simplified approach of Zuckerberg's method
in order to give convex-hull proofs for special cases of
the Boolean quadric polytope (see \cite{padberg1989boolean})
with multiple-choice constraints.

\paragraph{Contribution}

In the present article, we aim to show the power and flexibility
of Zuckerberg's approach to conduct convex-hull proofs.
To this end, we give an even more concise and accessible derivation
of the technique and relate it to the method in its original form.
Our novel way to introduce the method
is based on so-called \emph{set characterizations},
which provide a structured way of devising the algorithmic parts
of the convex-hull proofs.
It directly relates the set-theoretic representations to be found
to the constraints determining the integer points within the polyhedron
to be analysed.
Most notably, we use this concept
to significantly increase the scope of Zuckerberg's method.
While the original method is only applicable to $0$/$1$-polytopes,
we extend it from binary polytopes to arbitrary, especially integer polyhedra
and even much more general convex sets.

We demonstrate the wide applicability of our set characterization framework
by reproving several known convex-hull results
for both binary and integral polyhedra.
To facilitate the design of Zuckerberg convex-hull proofs,
we connect these examples with the introduction
of three basic proof strategies,
namely greedy placement, feasibility subproblems and transformation.
Altogether, this allows us to give simple constructions to represent a fractional point
in a given polyhedron as a convex combination of its vertices
where this was not straightforward before.
Moreover, we give further extensions of the method
to enable convex-hull proofs for function graphs over polytopes.
On the one hand, these extensions allow to prove convex-hull descriptions
for graphs of Boolean functions over $0$/$1$-polytopes.
On the other hand, they can be applied to bilinear functions over arbitrary polytopes,
generalizing the result from \cite{gupte2020extended} for bilinear functions over unit-boxes.
In summary, we show that Zuckerberg's method is a valuable tool
for conducting convex-hull proofs.
At the same time, our extensions of the framework
even allow to use it in much more general cases.

\paragraph{Structure}
This article is structured as follows.
We start by giving a detailed introduction
to Zuckerberg's proof technique for $0$/$1$-polytopes in Section~2.
We also establish our framework of set characterizations
for geometric convex-hull proofs.
Section~3 features three indicative examples of its application.
Each example highlights a novel algorithmic strategy
to conduct Zuckerberg-type convex-hull proofs.
In Section~4, we generalize Zuckerberg's method to arbitrary convex sets
by passing from one-dimensional set-theoretic representations to two-dimensional ones.
In particular, we will derive new techniques for convex-hull proofs
for the case of integer polyhedra.
Analogously to the binary case,
Section~5 gives examples for the use of our extended technique
in the context of mixed-binary optimization problems.
Among others, we show how to use the scheme
to prove total unimodularity of the constraint matrices
of combinatorial problems.
In Section~6, we derive further extensions of our approach
which allow to give convex-hull proofs for the graphs of Boolean and bilinear functions
over polytopal domains and introduce a generalized framework of set characterizations
for this purpose.
Our conclusions can be found in Section~7.
Finally, in the online supplement \cite{supplement} to this article,
we provide several further examples
for the application of our framework in the context of stable-set problems,
mixed-integer models for piecewise linear functions as well as interval matrices
and give some proofs omitted in Section~6.
\paragraph{Notation}
To facilitate notation, we denote the power set of a set $A$ by $\mP(A)$. Further, we write $[n]$ for the set $\set{1,\ldots,n} $ for any $n\in \N$. Especially $[0] \coloneqq \emptyset$. 
\section{Geometric convex-hull proofs for $0$/$1$-polytopes}
\label{Sec:Geometric_Proofs}

In this section, we revisit Zuckerberg's method for convex-hull proofs
for combinatorial decision or optimization problems
(see \cite{zuckerberg2016geometric,bienstock2004subset}).
We start by briefly summarizing it,
based on the condensed version of the method
that was derived in \cite{gupte2020extended}.
Then we introduce the concept of set characterizations
to significantly simplify the derivation of the set construction algorithms
which form the core of Zuckerberg-type convex-hull proofs.
Furthermore, we give set characterizations for many types of constraints
which typically occur in combinatorial optimization
and give some first indicative examples for their practical use.
Finally, we put our new approach into context with the original framework by Zuckerberg
to highlight how much simpler convex-hull proofs can now be conducted.

Consider a $0$/$1$-polytope $ P \coloneqq \conv(\FF) $
with vertex set $ \FF \subseteq \F^n $
together with a second polytope $ H \subseteq \R^n $
which is given via an inequality description.
If we want to prove $ P = H $, we can proceed
by verifying both $ \FF \subseteq H $ and $ H \subseteq P $.
The first inclusion is typically easy to show;
for the latter we can use Zuckerberg's method,
as outlined in the following.

Define $ U \coloneqq [0, 1) $, let $ \LL $ be the set of all unions
of finitely many half-open subintervals of $U$,
and let $ \mu $ be the Lebesgue measure restricted to $ \LL $,
that is
\begin{alignat*}{1}
	\LL \coloneqq \SSet{
		\bigcup\limits_{i = 1}^k [a_i, b_i)
		}
		{k \in \N \wedge 0 \leq a_1 < b_1 < a_2 < b_2 < \ldots < a_k < b_k \leq 1},\\
	\mu(S) \coloneqq \sum_{i = 1}^k (b_i - a_i)
		\text{ for any } S = [a_1, b_1) \cup \ldots \cup [a_k, b_k) \in \LL.
\end{alignat*}
Consider now the indicator function $ \func{\phi}{U \times \LL}{\F} $,
\begin{equation*}
	\phi(t, S) \coloneqq \begin{cases}
		1 & \text{ if } t \in S,\\
		0 & \text{ otherwise},
	\end{cases}
\end{equation*}
and let 
$ \func{\varphi}{U \times \LL^n}{\F^n},
	\varphi(t, S_1, \ldots, S_n) = v $,
where $ v_i \coloneqq \phi(t, S_i) $ for $ i \in [n] $.
In other words, $ \varphi $ maps the sets which are active at a certain $ t \in U $
onto the corresponding incidence vector in $ \F^n $.

The following result uses the above formalism to give a concise criterion
for~$H$ being a complete polyhedral description of $ \conv(\FF) $.
\begin{theorem}[{\cite[Theorem~4]{gupte2020extended}}, Zuckerberg's convex-hull characterization]
	Let $ \FF \subseteq \F^n $ and $ h \in \FR^n $.
	Then we have $ h \in \conv(\FF) $
	iff there are sets $ S_1, \ldots, S_n \in \LL $
	such that both $ \mu(S_i) = h_i $ for all $ i \in [n] $
	and $ \varphi(t, S_1, \ldots, S_n) \in \FF $
	for all $ t \in U $.
	\label{thm:gupteC}
\end{theorem}
\cref{thm:gupteC} provides a certificate for a point $ h \in H $
to be in $ \conv(\FF) $.
Thus, if we can find sets $ S_1, \ldots, S_n $ as required by~\cref{thm:gupteC}
for each point $ h \in H $, we have shown $ H \subseteq P $ as well.
Using the above framework even allows us to write a point $ h \in H $
as a convex combination of points in $ \FF $,
as the following corollary tells us.
This allows the spanning vertices to be used in heuristics, for example.
To this end, we define
\begin{equation*}
	L_{\xi}(S_1,\dotsc,S_n) \coloneqq \SSet{t\in U}{\varphi(t, S_1, \ldots, S_n) = \xi}.
\end{equation*}
to denote the support of a each vertex $ \xi \in \FF $ in $U$.
\begin{corollary}[Convex combinations]
	\label{Cor:Convex_Combinations}
	Under the same assumptions as in \cref{thm:gupteC},
	let $ \lambda_\xi \coloneqq \mu(L_{\xi}(S_1,\dotsc,S_n)) $
	for each $ \xi \in \FF $.
	Then we have $ h = \sum_{\xi \in \FF} \lambda_\xi \xi $,
	$ \sum_{\xi \in \FF} \lambda_\xi = 1 $
	and $ \lambda_\xi \geq 0 $ for all $ \xi \in \FF $.
\end{corollary}
The above corollary was not stated explicitly in \cite{gupte2020extended},
but it is one direction of the proof of Theorem~4 therein.
We already remark here that both \cref{thm:gupteC} and \cref{Cor:Convex_Combinations}
are special cases of the results we will prove in \cref{sec:zuckerberg_extended}
for general convex sets (and integer polyhedra in particular).

In combinatorial optimization, the vertex set $ \FF $ is typically implicitly defined
via an inequality description of the feasible incidence vectors
of the underlying problem.
We will now show that based on such a description,
we can make the expression $ \varphi(t, S_1, \ldots, S_n) \in \FF $
in \cref{thm:gupteC} more concrete.
For this purpose, we translate each constraint defining $ \FF $
into a logic statement of the following form.
\begin{definition}[Set characterization of a constraint]
	\label{def:set-char}
	Let $ \func{f}{\F^n}{\R} $, let $ b \in \R $,
	and let $ S_1, \ldots, S_n \in \LL $.
	The \emph{set characterization} of some constraint $ f(x) \leq b $
	is the following logic statement:
	\begin{equation*}
		f(\phi(t, S_1), \ldots, \phi(t, S_n)) \leq b \text{ holds for all } t \in U.
	\end{equation*}
\end{definition}
Note that this definition allows for arbitrary constraints on the incidence vectors,
not only linear ones.
We now observe that if $ \FF $ is given by such an implicit outer description,
we need to satisfy all set characterizations of the corresponding constraints
to fulfil the requirements of \cref{thm:gupteC} and \cref{Cor:Convex_Combinations}.
\begin{lemma}
	\label{Lem:Set_Characterization}
	Let $ \FF \coloneqq \set{x \in \F^n \mid f_j(x) \leq b_j\, \forall j \in [m]} $
	for some $ m \in \N $.
	Further, let $ P \coloneqq \conv(\FF) $,
	and let~$ H \subseteq \FR^n $ be some polytope.
	We have $ H = P $ iff both $ \FF \subseteq H $ holds
	and for each $ h \in H $ there are sets $ S_1, \ldots, S_n \in \LL $
	with $ \mu(S_i) = h_i $ for all $ i \in [n] $
	which satisfy the set characterization
	for each constraint $ f_j(x) \leq b_j $, $ j \in [m] $.
\end{lemma}
If some concrete function $f$ is given, along with some $ b \in \R $,
then the set characterization for the constraint $ f(x) \leq b $
given in \cref{def:set-char} can be simplified in many cases.
To give a first example, take the constraint $ x_1 \leq x_2 $
for some binary variables $ x_1, x_2 \in \F $.
Its set characterization reads
\begin{equation*}
	\phi(t, S_1) \leq \phi(t, S_2) \quad \forall t \in U.
\end{equation*}
Recalling the definition of $ \phi $,
this says that if for some $ t \in U $ the condition $ t \in S_1 $ holds,
then $ t \in S_2 $ follows.
So we can equivalently state the set characterization as $ S_1 \subseteq S_2 $.

For many common combinatorial constraints,
we have derived corresponding simplified set characterizations,
which are displayed in \cref{tb:set-char}.
\begin{table}[h]
	\centering
	\caption{Simplified set characterizations
		for combinatorial constraints with coefficients from $ \set{-1, 0, 1} $.}
	\label{tb:set-char}
	\begin{tabular}{cc}
		\toprule
		Constraint & Set characterization\\
		\midrule
		$ x_i \leq y_j $ & $ S_i \subseteq S_j $\\
		$ x_i \geq y_j $ & $ S_i \supseteq S_j $\\
		$ x_i = y_j $ &  $ S_i = S_j $\\
		\midrule
		$ \sum_{i \in I} x_i \leq 1 $ & $ S_i \cap S_j = \emptyset
			\quad \forall i, j \in I,\, i \neq j $\\
		$ \sum_{i \in I} x_i \geq 1 $ & $ \cup_{i \in I} S_i = U $\\
		$ \sum_{i \in I} x_i = 1 $ & $ \cup_{i \in I} S_i = U,\, S_i \cap S_j = \emptyset
			\quad \forall i, j \in I,\, i \neq j $\\
		\midrule
		$ \sum_{i \in I} x_i \leq k $ & $ \card{\set{i \in I \mid t \in S_i}} \leq k
			\quad \forall t \in U $\\
		$ \sum_{i \in I} x_i \geq k $ & $ \card{\set{i \in I \mid t \in S_i}} \geq k
			\quad \forall t \in U $\\
		$ \sum_{i \in I} x_i = k $ & $ \card{\set{i \in I \mid t \in S_i}} = k
			\quad \forall t \in U $\\
		\midrule
		$ \sum_{i \in I} x_i \leq \sum_{j \in J} y_j $ &
			$ \card{\set{i \in I \mid t \in S_i}} \leq \card{\set{j \in J \mid t \in S_j}}
				\quad \forall t \in U $\\
		$ \sum_{i \in I} x_i \geq \sum_{j \in J} y_j $ &
			$ \card{\set{i \in I \mid t \in S_i}} \geq \card{\set{j \in J \mid t \in S_j}}
				\quad \forall t \in U $\\
		$ \sum_{i \in I} x_i = \sum_{j \in J} y_j $ &
			$ \card{\set{i \in I \mid t \in S_i}} = \card{\set{j \in J \mid t \in S_j}}
				\quad \forall t \in U $\\
		\midrule
		$ x_i y_j = z_{ij} $ & $ S_i \cap S_j = S_{ij} $\\
		\bottomrule
	\end{tabular}
\end{table}
The set characterizations of the constraints defining~$P$
as in \cref{Lem:Set_Characterization}
provide hints on how to effectively design the sets $ S_1, \ldots, S_n $
as we will see in the following indicative examples.

\subsection{Connection between set characterization and algorithmic set construction}
\label{sec:mc-cormick1}
We consider the McCormick-linearization of a bilinear term as a first example
to illustrate the use of set characterizations within convex-hull proofs.
The example also illustrates that the set characterizations typically depend
on the inequality description of $ \FF $.
Let
\begin{equation*}
	H \coloneqq \set{(x, y, z) \in \FR^3 \mid
		z \geq 0,\, z \leq x,\, z \leq y,\, x + y - z \leq 1}.
\end{equation*}
We will compare the following two possible representations of the integral points in~$H$:
\begin{alignat}{1}
	\FF_1 &\coloneqq \set{(x, y, z) \in \F^3 \mid z \leq x, z \leq y, x + y - z \leq 1},\label{FF-1}\\
	\FF_2 &\coloneqq \set{(x, y, z) \in \F^3 \mid xy = z}. \label{FF-2}
\end{alignat}
In \cref{FF-2}, one single non-linear constraint
replaces the three linear constraints in~\cref{FF-1}. 
For each constraint in the two representations,
we need to derive a set characterization.
We can directly take them from \cref{tb:set-char}:
\begin{equation*}
	S_z \subseteq S_x, \quad S_z \subseteq S_y, \quad S_x \cap S_y \subseteq S_z
	\label{S-1}
\end{equation*}
for $ \FF_1 $ and
\begin{equation}
	S_x \cap S_y = S_z
	\label{S-2}
\end{equation}
for $ \FF_2 $.
One directly sees that both set characterizations are equivalent.
However, the second one is more compact.
In both cases, the sets need to have Lebesgue measures
equalling the coordinates of the arbitrary point $ h \in H $ to represent
and need to satisfy the set characterizations of the constraints defining the vertex set.
Throughout this article, we will give the convex-hull proofs via Zuckerberg's method
mainly in the form of algorithmic schemes to define sets
fulfilling these two conditions of \cref{Lem:Set_Characterization}.
As we will see, all these algorithms can be illustrated
via diagrams depicting the constructed sets in a coordinate system.

The construction rule for the sets in the McCormick-example
is given via the routine \textsc{Define-McCormick-Subsets} in \cref{fig:mc-cormick}.
\begin{figure}[h]
    \centering
	\begin{subfigure}{.5\textwidth}
		\centering
		\begin{algorithmic}[1]
			\Function{Define-McCormick-Subsets}{}
				\State $ S_x \coloneqq [0, h_x) $
				\State $ S_y \coloneqq [h_x - h_z, h_x - h_z + h_y) $
				\State $ S_z \coloneqq [h_x - h_z, h_x) $
			\EndFunction
		\end{algorithmic}
    \end{subfigure}
    \hfil
    \begin{subfigure}{.45\textwidth}
		\centering
		\begin{tikzpicture}[xscale=4,yscale=1]
			\draw[thick] (0,0) -- (1,0);
			\draw[thick] (0,.1) -- (0,.-.1);
			\draw (.2,.05) -- (.2,-.05);
			\draw (.4,.05) -- (.4,-.05);
			\draw (.6,.05) -- (.6,-.05);
			\draw (.8,.05) -- (.8,-.05);
			\draw[thick] (1,.1) -- (1,.-.1);
			\node at (0,-.4) {$0$};
			\node at (1,-.4) {$1$};
			
			\newcommand\AONE{.5*0}
			\newcommand\ATWO{.5*1}
			\newcommand\ATHREE{.5*2}
			
			\newcommand\BLOCKHEIGHT{0.4}
			\draw[fill=red!30] (0,0.2+\AONE) rectangle ++(0.3,\BLOCKHEIGHT);
			\draw[fill=blue!30] (0.3,0.2+\AONE) rectangle ++(0.2,\BLOCKHEIGHT);
			\draw[fill=blue!30] (0.3,0.2+\ATWO) rectangle ++(0.2,\BLOCKHEIGHT);
			\draw[fill=green!30] (0.5,0.2+\ATWO) rectangle ++(0.5,\BLOCKHEIGHT);
			\draw[fill=blue!30] (0.3,0.2+\ATHREE) rectangle ++(0.2,\BLOCKHEIGHT);

			\node at (-.1,.4+\AONE) {$S_{x}$};
			\node at (-.1,.4+\ATWO) {$S_{y}$};
			\node at (-.1,.4+\ATHREE) {$S_{z}$};
		\end{tikzpicture}  
    \end{subfigure}
	\caption{Routine \textsc{Define-McCormick-Subsets} (top),
		exemplary construction for the point~$h$
		with $ (h_x, h_y, h_z) = (0.5, 0.7, 0.2) $.
		The solution can be written as a convex combination
		of $ h = 0.3 (1, 0, 0) + 0.2 (1, 1, 1) + 0.5 (0, 1, 0) $.
		Those parts of the sets that belong to the same vertices
		are marked with the same colour.
	}
	\label{fig:mc-cormick}
\end{figure}
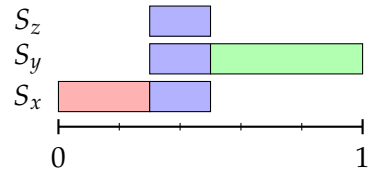
Based on representation~\cref{S-2}, it places~$ S_z $
such that it exhausts the total overlap of~$ S_x $ and~$ S_y $.
By construction, $ \mu(S_x) = h_x $, $ \mu(S_y) = h_y $ and $ \mu(S_z) = h_z $
hold for all $ h \in H $.
The inequalities in the definition of~$H$ further ensure
that the so-defined sets are all subsets of~$U$
This finishes the proof of $ H = \conv\set{(x, y, z) \in \F^3 \mid xy = z} $.

Once the sets for the given point~$h$ are constructed,
\cref{Cor:Convex_Combinations} tells us how to derive the coefficients
to express~$h$ as a convex combination of the vertices of~$H$.
The latter are given by $ \xi_1 \coloneqq (0, 0, 0) $,
$ \xi_2 \coloneqq (1, 0, 0) $, $ \xi_3 \coloneqq (0, 1, 0) $
and $ \xi_4 \coloneqq (1, 1, 1) $ in our example.
Each point $ t \in U $ is now mapped to some vertex~$ \xi_t $ of~$H$
via the mapping~$ \varphi $.
By measuring the union of all points that map to a certain vertex,
we can derive the coefficient for this vertex.
For the routine \textsc{Define-McCormick-Subsets}, we obtain
\begin{alignat*}{1}
	\mu(L_{\xi_1}(S_1, \ldots, S_n)) &= \mu([h_x - h_z + h_y, 1)) = 1 - h_y - h_x + h_z,\\
	\mu(L_{\xi_2}(S_1, \ldots, S_n)) &= \mu([0, h_x - h_z)) = h_x - h_z,\\
	\mu(L_{\xi_3}(S_1, \ldots, S_n)) &= \mu([h_x, h_x - h_z + h_y)) = h_y -h_z,\\
	\mu(L_{\xi_4}(S_1, \ldots, S_n)) &= \mu([h_x - h_z, h_x)) = h_z.
\end{alignat*}
Thus, we know $ h = (1 - h_x - h_y + h_z) \xi_1
	+ (h_x - h_z) \xi_2
	+ (h_y -h_z) \xi_3
	+ h_z \xi_4 $,
\cf the example given in \cref{fig:mc-cormick}.
 
\subsection{Non-uniqueness of set representations}

In a second example, we illustrate that the choice of the set construction
used for \cref{Lem:Set_Characterization}
determines which vertices are used to write a point $ h \in H $ as a convex combination
of vertices in $ \FF $.
In particular, this choice is not unique.

Consider the two-dimensional unit-box $ H \coloneqq \FR^2 $ and take $ \FF \coloneqq \F^2 $.
As there are no constraints on the binary points in $ \FF $,
no set characterization needs to hold.
We thus only have to fulfil the measure criteria.
\begin{figure}[h]
    \centering
	\begin{subfigure}{.5\textwidth}
		\begin{algorithmic}[1]
			\Function{Define-Box-Subsets-A}{}
			\State $ S_x \coloneqq [0, h_x) $
			\State $ S_y \coloneqq [0, h_y) $
			\EndFunction
		\end{algorithmic}
	\end{subfigure}
	\hfil
	\begin{subfigure}{.45\textwidth}
		\centering
		\begin{tikzpicture}[xscale=4,yscale=1]
			\draw[thick] (0,0) -- (1,0);
			\draw[thick] (0,.1) -- (0,.-.1);
			\draw (.2,.05) -- (.2,-.05);
			\draw (.4,.05) -- (.4,-.05);
			\draw (.6,.05) -- (.6,-.05);
			\draw (.8,.05) -- (.8,-.05);
			\draw[thick] (1,.1) -- (1,.-.1);
			\node at (0,-.4) {$0$};
			\node at (1,-.4) {$1$};
			
			\newcommand\AONE{.5*0}
			\newcommand\ATWO{.5*1}
			
			\newcommand\BLOCKHEIGHT{0.4}
			\draw[fill=blue!30] (0,0.2+\AONE) rectangle ++(0.5,\BLOCKHEIGHT);
			\draw[fill=blue!30] (0,0.2+\ATWO) rectangle ++(0.5,\BLOCKHEIGHT);

			\node at (-.1,.4+\AONE) {$S_{x}$};
			\node at (-.1,.4+\ATWO) {$S_{y}$};
		\end{tikzpicture} 
    \end{subfigure}
	\par\vskip\floatsep
	\begin{subfigure}{.5\textwidth}
		\begin{algorithmic}[1]
			\Function{Define-Box-Subsets-B}{}
			\State $ S_x \coloneqq [0, h_x) $
			\State $ S_y \coloneqq [1 - h_y, 1) $
			\EndFunction
		\end{algorithmic}
	\end{subfigure}\hfil
	\begin{subfigure}{.45\textwidth}
		\centering
		\begin{tikzpicture}[xscale=4,yscale=1]
		\draw[thick] (0,0) -- (1,0);
		\draw[thick] (0,.1) -- (0,.-.1);
		\draw (.2,.05) -- (.2,-.05);
		\draw (.4,.05) -- (.4,-.05);
		\draw (.6,.05) -- (.6,-.05);
		\draw (.8,.05) -- (.8,-.05);
		\draw[thick] (1,.1) -- (1,.-.1);
		\node at (0,-.4) {$0$};
		\node at (1,-.4) {$1$};
		
		\newcommand\AONE{.5*0}
		\newcommand\ATWO{.5*1}
		
		\newcommand\BLOCKHEIGHT{0.4}
		\draw[fill=green!30] (0,0.2+\AONE) rectangle ++(0.5,\BLOCKHEIGHT);
		\draw[fill=orange!30] (0.5,0.2+\ATWO) rectangle ++(0.5,\BLOCKHEIGHT);
		
		\node at (-.1,.4+\AONE) {$S_{x}$};
		\node at (-.1,.4+\ATWO) {$S_{y}$};
		\end{tikzpicture}
	\end{subfigure}
	\caption{The two routines \textsc{Define-Box-Subsets-A} (left top)
		and \textsc{Define-Box-Subsets-B} (left bottom)
		together with exemplary constructions for the point~$h$
		with $ (h_x, h_y) = (0.5, 0.5) $ for \textsc{Define-Box-Subsets-A} (right top)
		and \textsc{Define-Box-Subsets-B} (right bottom).
		Routine \textsc{Define-Box-Subsets-A} results in the representation
		$ h = 0.5 (1, 1) + 0.5 (0,0 ) $, while \textsc{Define-Box-Subsets-B} yields
		$ h = 0.5 (1, 0) + 0.5 (0, 1) $.
		Those parts of each set which belong to the same vertices
		in the convex combination representing $h$ are marked with the same colour.
	}
	\label{fig:box}
	\null\vfill\null
\end{figure}
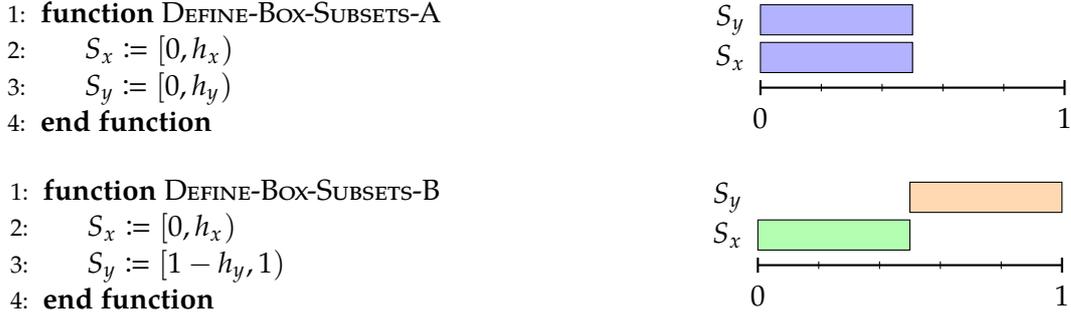
In \cref{fig:box}, we give two different construction rules for the sets $ S_x $ and $ S_y $
via the routines \textsc{Define-Box-Subsets-A} and \textsc{Define-Box-Subsets-B}.
Note that the definition of~$H$ ensures
that the sets~$ S_x $ and~$ S_y $ are always subsets of~$U$.
Both routines define valid choices for the two sets for each point $ h \in H $.
However, the resulting convex combinations of~$h$ via vertices in~$ \FF $
obtained via \cref{Cor:Convex_Combinations} are different from each other.

\subsection{Connection to the original method}

Zuckerberg's method for proving convex-hull characterizations
was first published in concise form in \cite{zuckerberg2016geometric},
although an antecedent had already appeared in his PhD thesis (see \cite{zuckerberg2004set}).
The main result is stated there in a very general form:
instead of choosing subsets of a real line segment as described above,
the sets could be chosen from an arbitrary measure space. 
This requires more complex definitions and notation.
We will shortly review Zuckerberg's original theorem here
to put our approaches into context
before we continue with and build upon the condensed version.
 
Using the same notation as above,
we are given a $0$/$1$-polytope $ P = \conv(\FF) $ with vertex set $ \FF \subseteq \F^n $
together with a second polytope~$H$, and the task is to prove $ H \subseteq P $.
According to Zuckerberg's original approach, we first need to represent $ \FF $
as a finite set-theoretic expression
consisting of unions, intersections and complements of the sets
\begin{equation*}
	A_i \coloneqq \SSet{a \in \F^n}{a_i = 1}, \quad i = 1, \ldots, n.
	\label{equ:zuck-a}
\end{equation*}
Let $ F(\set{A_i}) $ be such a representation of~$ \FF $.
Note that this is possible for any~$ \FF $ as we can always choose 
$ \FF = F_1(A_1, \ldots, A_n) \coloneqq
	\cup_{v \in \FF} ((\cap_{i \in [n]:\, v_i = 1} A_i) \cap
		(\cap_{i \in [n]:\, v_i = 0} \bar{A}_i)) $,
where $ \bar{A}_i $ denotes the complement of~$ A_i $ (in $ \F^n $).
Zuckerberg's original result can now be stated as follows.
\begin{theorem}[{\cite[Theorem 7]{zuckerberg2016geometric}}]
	Let $ \FF \subseteq \F^n $, and let $ F(\set{A_i}) $
	be a set-theoretic expression
	of finitely many unions, intersections and complementations
	of sets from $ \set{A_1, \ldots, A_n} $ such that $ F(A_1, \ldots, A_n) = \FF $.
	Further, let $ \QQ = (\hat{U}, \hat{\LL}) $ be any algebra
	with a basic set~$ \hat{U} $ and a family~$ \hat{\LL} $
	of subsets of~$ \hat{U} $,
	and let $ \Xi $ be any probability measure on $ \QQ $.
	Then $ x \in \FR^n $ belongs to $ \conv(\FF) $
	if there are sets $ S_i \in \bar{\LL} $, $ i = 1, \ldots, n $
	with $ x_i = \Xi(S_i) $ for all~$i$ and $ \Xi(F(\set{S_i})) = 1 $.
	\label{thm:zuck_original}
\end{theorem}
In order to use \cref{thm:zuck_original}, we first need to find
a set-theoretic expression to represent $ \FF $.
While the representation~$ F_1 $ is always possible,
it is not helpful, since it does not allow to easily derive
criteria for how to find suitable sets $ S_i $.
For instance, for the McCormick-example in \cref{sec:mc-cormick1},
the vertex set can be written as $ \FF = A_x \cap A_y \Leftrightarrow A_z $.
One can easily verify that if the sets $ S_x $, $ S_y $ and $ S_z $
satisfy condition~\cref{S-2}, namely $ S_x \cap S_y = S_z $,
then $ \Xi(F(\set{S_i})) = 1 $ holds.
Conversely though, there is no straightforward way to the derive set characterizations
from the set-theoretic expression~$F$.
The possibility to directly derive set characterizations
from the constraints defining~$ \FF $, however,
significantly reduces the effort to conduct Zuckerberg convex-hull proofs
and is only given in the simplified version.
To introduce this concept is therefore one of the main contributions of this article.

The simplified version of Zuckerberg's results we build on
was introduced in \cite{gupte2020extended}
by choosing $ \hat{U} = U $, $ \hat{\LL} = \LL $ and $ \Xi = \mu $.
The condition $ \Xi(F(\set{S_i})) = 1 $ can then be replaced by $ F(\set{S_i}) = U $.
The authors also show that this allows to drop the set-theoretic expression~$F$ entirely
and further allows to replace $ F(\set{S_i}) = U $
with $ \varphi(t, S_1, \ldots, S_n) \in \FF $ for all $ t \in U $.
Their main result is then \cref{thm:gupteC} from above.

The real line is probably the simplest possible choice
for the measure space in \cref{thm:zuck_original},
and via \cref{thm:gupteC} it has the same expressive power as any other measure space.
Thus, on the one hand, the choice of more complex measure spaces
might allow for easier-to-state convex-hull proofs in certain cases
(which Zuckerberg himself states as an avenue for future research).
On the other hand, however, the real line is sufficient
to prove a vast variety of results, as the examples in the following section
as well as those provided in \cite{zuckerberg2016geometric,gupte2020extended,harris2020convex,BMS2020} show.
Furthermore, it allows for a much more lightweight notation
and enables us to use the concept of set characterizations we have introduced above.
Finally, this concise form will enable us to derive several significant extensions
of Zuckerberg's approach, in particular a proof technique applicable to general convex sets
and criteria for convex-hull proofs for graphs of certain functions over polytopal domains.

\section{Set characterizations and proof strategies for binary problems}

In the following, we will show how to use our concept of set characterizations
to give Zuckerberg convex-hull proofs for more complex $0$/$1$-polytopes.
We do this by reproving several known, popular results
to demonstrate how set characterizations
help define the sets~$ S_i $ for \cref{Lem:Set_Characterization}.
The order in which to define these sets is highly problem specific.
We will see that very often a certain \qm{natural} ordering can be used
to successfully conduct convex-hull proofs.
In an example involving the shortest-path problem,
we will use a topological ordering of the nodes of the underlying graph.
The second example for a certain set-packing problem
shows how to exploit a depth-first-search on a tree.
It will also turn out here that we can use Zuckerberg's method to compute
the vertices spanning a point inside the polytope,
which was not straightforward to do beforehand.
And in the last example, where we consider the odd-hole inequality
for the stable-set problem, we follow neighbourly nodes along the underlying cycle.
These examples are representative for three promising general strategies
to define the sets~$ S_i $.
The first one is a greedy strategy which places the sets according to local criteria.
The second strategy extracts the placement of a group of sets
from the solution of an auxiliary optimization problem.
Finally, the third strategy transforms the point $ h \in H $
to an auxiliary point~$ \bar{h} \in H $
for which the placement of the sets is easier,
and afterwards retransforms the sets in order to the express the original point.

The core of a Zuckerberg convex-hull proof
is an algorithmic scheme to define the sets required in \cref{Lem:Set_Characterization}.
To this end, we first define the subroutine \textsc{Match}
in \cref{fig:def_match}.
It is useful in problems where the feasible set of binary points
is constrained by multiple-choice constraints. 
Its inputs are a set $ S \in \LL $
together with a list of diameters $ (w_1, \ldots, w_k) \in \R_+^n $
for some $ k \geq 1 $.
The output is then a list of subsets $ (S_1, \ldots, S_k) $ of~$S$
with $ \mu(S_i) = w_i $ for all $ i \in [k] $.
If $ w_1 + \ldots + w_k \leq \mu(S) $ holds, these subsets are pairwise disjoint
(\cf the set characterization for a multiple-choice constraint
stated in \cref{tb:set-char}).
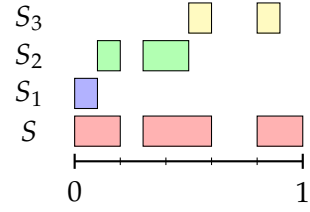
\begin{figure}[h]
	\hfill
	\begin{subfigure}[b]{0.58\textwidth}
		\begin{algorithmic}[1]
			\Require{$ S \in \LL,
				(w_1, \ldots, w_k), w_i \in [0, 1), i \in [k] $}
			\Ensure{$ (S_1, \ldots, S_k) $ with
				$ S_i \subseteq S $ for $ i \in [k] $. 
				If $ w_1 + \ldots + w_k \leq \mu(S) $ then
				$ S_i \cap S_j = \emptyset $ for $ i, j \in [k] $
				with $ i \neq j $}
			\Function{Match}{$ S, (w_1, \ldots, w_k) $}
			\State $ t_0 \leftarrow 0 $
			\For{$ r = 1, \ldots, k $}
			\State $ t_r \leftarrow
				\min\set{t \in S \mid \mu(S \cap [t_{r - 1}, t]) = w_r} $
			\If{$t_r < +\infty$}
				\State $ S_r \coloneqq S \cap [t_{r - 1}, t_r) $
			\Else
				\State $ t_r \leftarrow
					\min\set{t \in S \mid \mu(S \cap [0, t])$\\
					\qquad\qquad\qquad\qquad\qquad\quad\qquad
						$ = \mu(S \cap [t_{r-1},1])} $ 
				\State $ S_r \coloneqq S \cap ([t_{r - 1}, 1) \cup [0, t_r)) $
			\EndIf
			\EndFor
			\State{\textbf{return} $ (S_1, \ldots, S_k) $}
			\EndFunction
		\end{algorithmic}
	\end{subfigure}
	\hfill
	\begin{subfigure}[b]{0.35\textwidth}
		\begin{tikzpicture}[xscale=6,yscale=1]
			\draw[thick] (0,0) -- (0.5,0);
			\draw[thick] (0,.1) -- (0,.-.1);
			\draw (.1,.05) -- (.1,-.05);
			\draw (.2,.05) -- (.2,-.05);
			\draw (.3,.05) -- (.3,-.05);
			\draw (.4,.05) -- (.4,-.05);
			\draw[thick] (.5,.1) -- (.5,.-.1);
			\node at (0,-.4) {$0$};
			\node at (.5,-.4) {$1$};
			\draw[fill=red!30] (0,0.2) rectangle ++(0.1,0.4);
			\draw[fill=red!30] (0.15,0.2) rectangle ++(0.15,0.4);
			\draw[fill=red!30] (0.4,0.2) rectangle ++(0.1,0.4);
			\draw[fill=blue!30] (0,0.2+.5*1) rectangle ++(0.05,0.4);
			\draw[fill=green!30] (0.05,0.2+.5*2) rectangle ++(0.05,0.4);
			\draw[fill=green!30] (0.15,0.2+.5*2) rectangle ++(0.1,0.4);
			\draw[fill=yellow!30] (0.25,0.2+.5*3) rectangle ++(0.05,0.4);
			\draw[fill=yellow!30] (0.4,0.2+.5*3) rectangle ++(0.05,0.4);
			
			\node at (-.1,.4+.5*0) {$\S$};
			\node at (-.1,.4+.5*1) {$\S_1$};
			\node at (-.1,.4+.5*2) {$\S_2$};
			\node at (-.1,.4+.5*3) {$\S_3$};
		\end{tikzpicture}
		\hfill\null
	\end{subfigure}
	\caption{Subroutine \textsc{Match} (left)
		and exemplary output for defining three subsets
		of some set~$S$ (right)}
	\label{fig:def_match}
\end{figure}

\subsection{The greedy strategy}
\label{sec:shortest-path}

In the greedy proof strategy,
we place the current set to be defined to the first spot
which satisfies all set characterizations,
without considering the subsequent sets to be placed.
When conducting Zuckerberg proofs,
this is generally the first strategy one should try.
This is because of its simplicity,
and if this strategy works, it typically leads to very short proofs.
We showcase the use of this technique
for the shortest-path polytope on an acyclic graph.

Let $ G = (V, A) $ be a connected, directed and acyclic graph (DAG).
The node set~$V$ contains two special nodes~$s$ and~$d$,
and the goal is to find a path from~$s$ to~$d$.
For ease of exposition, the node~$s$ shall only have outgoing arcs,
while~$d$ only has incoming arcs.
The set of feasible paths can be represented by introducing a binary variable $ x_a \in \F $
for each $ a \in A $ to model the choice of arcs
together with the following system of linear constraints:
\begin{alignat}{1}
	\sum_{a \in \outedges{s}} x_a &= 1 \label{shortpath:1}\\
	\sum_{a \in \outedges{v}} x_a - \sum_{a \in \inedges{v}} x_a &= 0 \quad \forall v \in V \setminus \set{s, d} \label{shortpath:2}\\
	\sum_{a \in \inedges{d}} x_a &= 1 \label{shortpath:3}\\
	0 \leq  x_a & \leq 1. \label{shortpath:4}
\end{alignat}
We now give a Zuckerberg-type proof for the well-known result
stating the integrality of the above system.
\begin{theorem}
	Let $ P \coloneqq \conv\set{x \in \F^{\card{A}} \mid
			\cref{shortpath:1,shortpath:2,shortpath:3,shortpath:4}} $
	be the shortest-path polytope
	and $ H \coloneqq \set{x \in \FR^{\card{A}} \mid
			\cref{shortpath:1,shortpath:2,shortpath:3,shortpath:4}} $
	its linear relaxation.
	Then we have $ P = H $.
\end{theorem}
\begin{proof}
It is obvious that $ P \subseteq H $.
In order to prove $ H \subseteq P $,
we need to transform
the constraints \cref{shortpath:1,shortpath:2,shortpath:3}
into set characterizations.
Referring to \cref{tb:set-char},
we can directly state them as follows:
\begin{alignat}{1}
	\card{\set{i \in \outedges{s} \mid t \in S_i}} = 1
		& \quad \forall t \in U,\label{st:char1}\\
	\card{\set{i \in \outedges{v} \mid t \in S_i}}
		= \card{\set{j \in \inedges{v} \mid t \in S_j}}
		& \quad \forall v \in V \setminus \set{s, d},
			\forall t \in U,\label{st:char2}\\
	\card{\set{i \in \inedges{d} \mid t \in S_i}} = 1
		& \quad \forall t \in U.\label{st:char3}
\end{alignat}
Note that inequalities~\cref{shortpath:4}
do not have a set characterization of their own above
as they are already implied by the fact
that all sets need to be subsets of $ U = [0, 1) $. 
Further, inequality~\cref{shortpath:3} is redundant
and only stated for better readability.
Therefore, set characterization~\cref{st:char3}
is already implied by~\cref{st:char1} and~\cref{st:char2}.

For each point $ h \in H $,
we now need to find sets~$ S_a $ for all $ a \in A $
such that $ \mu(S_a) = h_a $
as well as set characterizations \cref{st:char1,st:char2,st:char3} hold.
The sets~$ S_a $ are defined
via the routine \textsc{Define-Shortest-Path-Subsets}
presented in \cref{fig:zuckerberg-stpath3}.
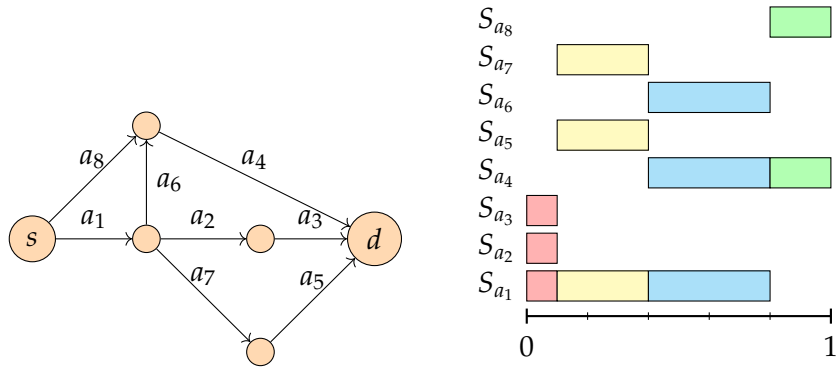
\begin{figure}
	\null\vfill\null
	\begin{algorithmic}[1]
		\Function{Define-Shortest-Path-Subsets}{}
			\For{$ v \in \textsc{TopologicalSort(V)} $}
				\If{$ v = s $}
					\State $ R \leftarrow [0, 1) $
				\Else
					\State $ R \leftarrow \cup_{a \in \inedges{v}} S_a $
				\EndIf
				\State Let $ (a_1, \ldots, a_p) $
					be any ordering of the arcs
					in $ \outedges{v} $, where $ p \coloneqq \card{\outedges{v}} $  
				\State $ (S_{a_1}, \ldots, S_{a_p}) \coloneqq
					\textsc{Match}(R, (h_{a_1}, \ldots, h_{a_p})) $
			\EndFor
		\EndFunction
	\end{algorithmic}
	\null\vfill\null
	\centering 
	\begin{tikzpicture}[xscale=1.5,yscale=1.5]
		\node[shape=circle,draw=black, fill=orange!30] (A) at (0,0) {$s$};
		\node[shape=circle,draw=black, fill=orange!30] (B) at (1,0) {};
		\node[shape=circle,draw=black, fill=orange!30] (C) at (2,0) {};
		\node[shape=circle,draw=black, fill=orange!30] (D) at (3,0) {$d$};
		\node[shape=circle,draw=black, fill=orange!30] (E) at (1,1) {};
		\node[shape=circle,draw=black, fill=orange!30] (F) at (2,-1) {};
		\path [->] (A) edge node[above] {$a_1$} (B);
		\path [->] (B) edge node[above] {$a_2$} (C);
		\path [->] (C) edge node[above] {$a_3$} (D);
		\path [->] (E) edge node[above] {$a_4$} (D);
		\path [->] (F) edge node[above] {$a_5$} (D);
		\path [->] (B) edge node[right] {$a_6$} (E);
		\path [->] (B) edge node[above] {$a_7$} (F);
		\path [->] (A) edge node[above] {$a_8$} (E);
	\end{tikzpicture}  
	\qquad
	\begin{tikzpicture}[xscale=4,yscale=1]
		\draw[thick] (0,0) -- (1,0);
		\draw[thick] (0,.1) -- (0,.-.1);
		\draw (.2,.05) -- (.2,-.05);
		\draw (.4,.05) -- (.4,-.05);
		\draw (.6,.05) -- (.6,-.05);
		\draw (.8,.05) -- (.8,-.05);
		\draw[thick] (1,.1) -- (1,.-.1);
		\node at (0,-.4) {$0$};
		\node at (1,-.4) {$1$};
		
		\newcommand\AONE{.5*0}
		\newcommand\ATWO{.5*1}
		\newcommand\ATHREE{.5*2}
		\newcommand\AFOUR{.5*3}
		\newcommand\AFIVE{.5*4}
		\newcommand\ASIX{.5*5}
		\newcommand\ASEVEN{.5*6}
		\newcommand\AEIGHT{.5*7}
	
		\newcommand\BLOCKHEIGHT{0.4}
		\draw[fill=red!30] (0,0.2+\AONE) rectangle ++(0.1,\BLOCKHEIGHT);
		\draw[fill=red!30] (0,0.2+\ATWO) rectangle ++(0.1,\BLOCKHEIGHT);
		\draw[fill=red!30] (0,0.2+\ATHREE) rectangle ++(0.1,\BLOCKHEIGHT);
		
		\draw[fill=yellow!30] (0.1,0.2+\AONE) rectangle ++(0.3,\BLOCKHEIGHT);
		\draw[fill=yellow!30] (0.1,0.2+\AFIVE) rectangle ++(0.3,\BLOCKHEIGHT);
		\draw[fill=yellow!30] (0.1,0.2+\ASEVEN) rectangle ++(0.3,\BLOCKHEIGHT);
		
		\draw[fill=cyan!30] (0.4,0.2+\AONE) rectangle ++(0.4,\BLOCKHEIGHT);
		\draw[fill=cyan!30] (0.4,0.2+\AFOUR) rectangle ++(0.4,\BLOCKHEIGHT);
		\draw[fill=cyan!30] (0.4,0.2+\ASIX) rectangle ++(0.4,\BLOCKHEIGHT);	
		
		\draw[fill=green!30] (0.8,0.2+\AEIGHT) rectangle ++(0.2,\BLOCKHEIGHT);
		\draw[fill=green!30] (0.8,0.2+\AFOUR) rectangle ++(0.2,\BLOCKHEIGHT);

		\node at (-.1,.4+\AONE) {$S_{a_1}$};
		\node at (-.1,.4+\ATWO) {$S_{a_2}$};
		\node at (-.1,.4+\ATHREE) {$S_{a_3}$};
		\node at (-.1,.4+\AFOUR) {$S_{a_4}$};
		\node at (-.1,.4+\AFIVE) {$S_{a_5}$};
		\node at (-.1,.4+\ASIX) {$S_{a_6}$};
		\node at (-.1,.4+\ASEVEN) {$S_{a_7}$};
		\node at (-.1,.4+\AEIGHT) {$S_{a_8}$};
	\end{tikzpicture}  
	\caption{Routine \textsc{Define-Shortest-Path-Subsets} (top),
		exemplary graph with $6$ nodes (bottom left)
		and possible output of the routine for the point
		$ h = (0.8, 0.1, 0.1, 0.6, 0.3, 0.4, 0.3, 0.2) $ (bottom right).
		There are four paths in the graph,
		namely $ (a_1, a_2, a_3) $, $ (a_8, a_4) $,
		$ (a_1, a_7, a_5) $ and $ (a_1, a_6, a_4) $.
		Those parts of the sets corresponding to a certain path
		are marked in the same colour.}
	\label{fig:zuckerberg-stpath3}
	\null\vfill\null
\end{figure}
The algorithm processes the nodes in the graph in topological order,
where \textsc{TopologicalSort} is any routine producing such an order.
In each iteration, it places the sets for all outgoing arcs of the current node
via a call to the routine \textsc{Match}.
This ensures that conditions \cref{st:char1,st:char2,st:char3} are satisfied.
By starting at node~$s$ and processing the nodes in topological order,
we are sure that once a node is reached all sets for the incoming arcs have been defined.
Finally, the make-up of subroutine \textsc{Match}
guarantees $ \mu(S_a) = h_a $ for all $ a \in A $.
Thus, we have proved $ H \subseteq P $.
\end{proof}
The greedy proof technique is most promising if the problem at hand
only features local constraints (like flow conservation or variable bounds)
as they allow to place the sets in consecutive fashion.
Constraints inducing global couplings between the variables make it harder to use.
In the \supplement, we give further examples
for the use of this technique in the context of clique and stable-set problems.
\subsection{Zuckerberg proofs via feasibility subproblems}

A further strategy for Zuckerberg proofs is to place groups of related sets simultaneously.
If the correct placement of these sets is too difficult to be stated explicitly,
it can be worthwhile to define an auxiliary optimization problem
from whose solution a feasible placement of the sets can be extracted.
It is then necessary to prove that this subproblem
is feasible for each point $ h \in H $ to be tested.
In case the optimization problem is a linear program,
one can try to use the Farkas lemma for the feasibility proof.
We highlight this technique at the hand of a polynomial-time solvable special case
of the clique problem with multiple-choice constraints.

Let $ G = (V, E) $ be an $m$-partite graph for some $ m \geq 1 $,
and let $ \VV = \set{V_1, \ldots, V_m} $ be the corresponding partition
of the node set~$V$.
The \emph{clique problem with multiple-choice constraints (CMPC)}
asks to find a clique of cardinality~$m$ in~$G$.
While it is NP-complete in general to decide if such a clique exists
(see \cite{cycle-free2020}),
there are several relevant special cases where this is possible in polynomial time.
These include CPMC under staircase compatibility (\cite{staircase2018})
and CPMC under a cycle-free dependency graph (\cite{cycle-free2020}).
The referenced works give complete convex-hull descriptions
for these two cases.

The CPMC polytope is the convex hull of all incidence vectors
of $m$-cliques in~$G$.
In the \supplement, we will reprove the result from \cite{staircase2018}
that staircase compatibility allows for totally unimodular formulations
of polynomial size for the CPMC polytope.
Here we consider the case where there are no cyclic dependencies
between the subsets~$ V_i $.
The authors of \cite{cycle-free2020} give a complete convex-hull description
for this case whose correctness they prove via the alternating colouration theorem
(see \cite{hoang1987alternating}).
Alternatively, they hint a proof via the strong perfect-graph theorem
(see \cite{strongperfect2006}).
In the following, we will give a much more elementary convex-hull proof
based on Zuckerberg's method which does neither use
alternating colourations nor perfectness.
In addition, we will be able to state the vertices
spanning any given point in the CPMC polytope,
for which there is no obvious derivation
using the approaches presented in \cite{cycle-free2020}.

Let $ \GG \coloneqq (\VV, \EE) $ with
\begin{equation*}
	\EE \coloneqq \lrset{\set{V_i, V_j} \subseteq \VV \mid
		(\exists u \in V_i)(\exists v \in V_j)\,
			\set{u, v} \notin E}
\end{equation*}
denote the \emph{dependency graph} of~$G$.
Note that $ \set{V_i, V_j} \in \EE $
is equivalent to the subgraph~$ \Gij $ induced by $ V_i \cup V_j $
not being a complete bipartite graph.
For ease of notation, we further define the \emph{neighbourhood} $ N_j(U) \subseteq V_j $
of a subset $ U \subseteq V $ in $ V_j $ as
\begin{equation*}
	N_j(U) \coloneqq \set{v \in V_j \mid
		(\exists u \in U)\, \set{u, v} \in E}.
\end{equation*}
It represents those nodes in $ V_j $
for which there is a compatible node in~$U$.

We will now show that the CPMC polytope is completely described
via the stable-set constraints and the trivial constraints
if the dependency graph is a forest.
\begin{theorem}(\cite[Theorem~3.1]{cycle-free2020})
	Let
	\begin{equation*}
		P(G, \VV) \coloneqq
			\conv\SSet{x \in \F^{\card{E}} }{ 
				\begin{split}
					\sum_{v \in V_i} x_v &= 1
						\quad \forall V_i \in \VV \\
				 	x_i + x_j &\leq 1
				 		\quad \forall \set{i, j} \notin E
				\end{split}}
	\end{equation*}
	be the CPMC polytope and
	\begin{equation*}
		H(G, \VV) \coloneqq
			\conv\SSet{x \in \FR^{\card{E}} }{ 
				\begin{split}
					\sum_{v \in V_i} x_v &= 1
						\quad \forall V_i \in \VV \\
					\sum_{v \in C} x_v &\leq 1
						\quad \forall \text{ stable sets } C \subseteq V
				\end{split}}
	\end{equation*}
	its stable-set relaxation.
	If $ \GG $ has no cycles, we have $ P(G, \VV) = H(G, \VV) $.
	\label{Thm:CPMCF}
\end{theorem}
\begin{proof}
The inclusion $ P(G, \VV) \subseteq H(G, \VV) $ holds trivially.
We now show the reverse inclusion.
The procedure to define the sets $ S_i $ for $ i \in V $
is given via the two routines \textsc{Define-CMPCF-Subsets}
and \textsc{Traverse-Tree} in~\cref{Fig:CPMCF}.
\begin{figure}[h]
	\null\vfill\null
	\begin{algorithmic}[1]
		\Function{Define-CMPCF-Subsets}{}
			\For{each tree in $ \GG $}
				\State Pick some $ V_i $ as the root node \Wlog
				\State Let $ (i_1, \ldots, i_p) $
					be any ordering of the elements in $ V_i $,
					where $ p \coloneqq \card{V_i} $
				\State $ (\S_{i_1}, \ldots, \S_{i_p})
					\coloneqq \textsc{Match}([0, 1),
						(h_{i_1}, \ldots, h_{i_p})) $
				\For{$ V_j \in N(V_i) $}
					\State \textsc{Traverse-Tree}($ V_i $, $ V_j $)
				\EndFor
			\EndFor
		\EndFunction
		\Function{Traverse-Tree}{$ V_i $, $ V_j $}
			\State $ \bar{x} \leftarrow $ Solve \cref{equ:feasibility_opt_cmpcf}
				with $ V_i $ and $ V_j $ and $h$ as input
			\State Let $ (j_1, \ldots, j_p) $
				be any ordering of the elements in $ V_j $,
				where $ p \coloneqq \card{V_j} $
			\For{$ j \in V_j $}
				\State $ S_j \coloneqq \emptyset $
			\EndFor	
			\For{$ i \in V_i $}
				\State $ (\hat{\S}_{j_1}, \ldots, \hat{\S}_{j_p})
					\leftarrow \textsc{Match}(S_i,
						(\bar{x}_{ij_1}, \ldots, \bar{x}_{ij_p})) $
				\For{$ j \in V_j $}
					\State $ S_j \coloneqq S_j \cup \hat{\S}_j $
				\EndFor	
			\EndFor
			\For{$ V_k \in N(V_j) $}
				\State \textsc{Traverse-Tree}($ V_j $, $ V_k $)
			\EndFor
		\EndFunction
	\end{algorithmic}
	\null\vfill\null
	\centering
	\caption{Routine \textsc{Define-CMPCF-Sets}}
	\label{Fig:CPMCF}
	\null\vfill\null
\end{figure}
The former routine iterates over all individual trees
in the dependency graph~$ \GG $.
In Line~3, it selects an arbitrary node (subset in the partition) $ V_i $
as the root node of the current tree.
Then it fixes an arbitrary ordering of the elements $ v \in V_i $
and places the corresponding sets next to each other
via a call to subroutine \textsc{Match} in Line~5.
Finally, it traverses the tree recursively in Lines~6--8
by calling the routine \textsc{Traverse-Tree},
whose input is a subset~$ V_i $
for which all sets have already been defined,
together with a set~$ V_j $, which is a neighbour of~$ V_i $.
The routine then places all sets for~$ V_j $. 
To do so, it solves a linear feasibility problem in Line~12
which is defined as follows:
the variables $ x_{ij} \in \R_+ $ encode
the measure of the overlap
between the sets~$ S_i $ and~$ S_j $.
These overlaps need to fulfil the set characterizations
\begin{alignat}{1}
	\card{\set{v \in V_i \mid t \in S_v}} &= 1
		\quad \forall V_i \in \VV, \forall t \in U,
		\label{Eq:CMMCF-MC-char1}\\
	\card{\set{v \in C \mid t \in S_v}} &\leq 1
		\quad \forall \text{ stable sets } C \subseteq V, \forall t \in U,
		\label{CPMCF-SS-char2}
\end{alignat}
which leads to the following linear programming system:
\begin{subequations}
\begin{alignat}{1}
	\sum_{j \in V_j} x_{ij} &= h_i \quad \forall i \in V_i,\\
	\sum_{i \in V_i} x_{ij} &= h_j \quad \forall j \in V_j,\\
	x_{ij} &\geq 0 \quad \forall (i,j) \in V_i \times V_j.
\end{alignat}
\label{equ:feasibility_opt_cmpcf}
\end{subequations}
\!\!In Lines~13--22, the routine chooses the sets for all elements in $ V_j $ accordingly.
It then proceeds recursively in Lines~23--25.

It remains to show that problem~\cref{equ:feasibility_opt_cmpcf}
is feasible for all $ h \in H $.
To prove this, we analyse its dual Farkas system,
which is given by
\begin{subequations}
\begin{alignat}{1}
	\sum_{i \in V_i} h_i y_i + \sum_{j \in V_j} h_j y_j &< 0,\\
	y_i + y_j &\geq 0 \quad \forall (i, j) \in V_i \times V_j.
\end{alignat}
\label{equ:feasibility_dual_cmpcf}
\end{subequations}
\!\!We will prove by contradiction that \cref{equ:feasibility_dual_cmpcf} has no solution
in order to show the feasibility of \cref{equ:feasibility_opt_cmpcf}.
To this end, consider some point $ h \in H $
and let $ \y $ be a corresponding solution
of \cref{equ:feasibility_dual_cmpcf}.
We first argue that we can assume $ \y \in \set{-1, 0, 1}^m $ w.l.o.g.
Via rescaling, we can assume that the lowest entry of~$ \y $ is~$-1$. 
Now let $ W \coloneqq \set{k \in V_i \times V_j \mid \y_k < 0} $.
We can assume $ \y_k = 0 $
for all elements in $ (V_i \times V_j) \setminus (W \cup N(W)) $, since this is always feasible if $ \y_k > 0 $ was feasible. 
Further, let $ P \coloneqq (p_1, \ldots, p_q) $
be a sorted list of the elements in $ \set{\y_k \in \R \mid k \in W} $
in decreasing order with $q \coloneqq \abs{P}$.
For $ p \in P $, let $ Q_p \coloneqq \set{k \in W \mid y_k = p} $,
and let $ R \coloneqq N(Q_{p_1}) \setminus N(Q_{p_2})
	\cup \ldots \cup N(Q_{p_q}) $.
Then check if $ \sum_{v \in R} h_v \y_v
	+ \sum_{v \in Q_{p_1}} h_v \y_v \geq 0 $ holds.
If yes, set $ \y_k = 0 $ for $ k \in Q_{p_1} \cup R $.
If no, set $ \y_k = p_2 $ for $ k \in Q_{p_1} $
and $ \y_k = -p_2 $ for $ k \in R $.
Now update $W$ and let $ P \coloneqq (p_1, \ldots, p_{q - 1}) $
again be a sorted list of the elements
in $ \set{\y_k \in \R \mid k \in W} $ in decreasing order.
This procedure lets~$P$ now contain
precisely one element less than before.
Repeat this until there is only one element in~$P$ left,
which has to be~$-1$,
so we can set $ \bar{y}_k = 1 $ for all $ k \in N(W) $.
This way, we have found an integral solution
to~\cref{equ:feasibility_dual_cmpcf}.
We then have 
\begin{alignat*}{1}
	\sum_{k \in N(W)} h_k \bar{y}_k + \sum_{k \in W} h_k \bar{y}_k &< 0,\\
	\sum_{k \in N(W)} h_k - \sum_{k \in W} h_k &< 0,\\
	1 - \sum_{k \in (V_i \times V_j) \setminus N(W)} h_k
		- \sum_{k \in W} h_k &< 0,\\
	\sum_{k \in (V_i \times V_j) \setminus N(W)} h_k
		+ \sum_{k \in W} h_k &> 1.
\end{alignat*}
However, this is impossible,
since the nodes $ ((V_i \times V_j) \setminus N(W)) \cup W $
form a stable set, which leads to a contradiction.
\end{proof}
Via \cref{Cor:Convex_Combinations}, this directly allows us
to represent a point $ h \in H $ as a convex combination of the vertices
of the CMPC polytope, which extends the results from \cite{cycle-free2020}.

This technique could be generalized by passing from linear
to more complex auxiliary problems to determine the placement of the sets.
The core of this proof technique consists in analysing the auxiliary problem
to verify its feasibility for any inputs arising within the algorithmic scheme.

\subsection{The transformation strategy}
\label{Sec:Transformation}

The third proof strategy we present makes use of the fact
that it can be easier to place the sets
for some points within a given polytope
than for others.
Thus, it is sometimes helpful to transform the arbitrary point
to be tested for membership in \cref{Lem:Set_Characterization}
to another, auxiliary point first.
Then, after placing the sets for this auxiliary point,
they are retransformed to represent the original point.
Such a transformation must respect the set characterizations of the vertex set.
We present this technique exemplarily for the convex hull
of all incidence vectors of stable sets in a single odd cycle.

The stable-set polytope of a graph $ G = (V, E) $
is defined as the convex hull of all vectors $ x \in \F^{\card{V}} $
that satisfy
\begin{equation}
	x_i + x_j \leq 1 \quad \forall (i, j) \in E.
	\label{stab:set-inq}
\end{equation}
If~$G$ is a cycle, the \emph{odd-cycle inequality}
\begin{equation}
	\sum_{i \in V} x_i \leq \frac{\card{V} - 1}{2}
\label{ineq:cycle}
\end{equation}
is valid for the corresponding stable-set polytope.
For an odd-cycle, it is sufficient to describe the complete convex hull,
together with inequalities~\cref{stab:set-inq} and the trivial inequalities.
\begin{theorem}
	Let $ G = (V, E) $ be an odd hole,
	let $ P(G) \coloneqq \set{x \in \F^{\card{V}} \mid
		\cref{stab:set-inq,ineq:cycle}} $
	be the stable-set polytope on~$G$,
	and let $ H(G) \coloneqq \set{x \in \FR^{\card{V}} \mid
		\cref{stab:set-inq,ineq:cycle}} $
	be its linear relaxation.
	Then we have $ P = H $.
\end{theorem}
\begin{proof}
	It is obvious that $ P(G) \subseteq H(G) $.
	For the converse, consider the set characterizations
	of \cref{stab:set-inq,ineq:cycle}, which are given by:
	\begin{alignat}{1}
		S_i \cap \S_j &= \emptyset
			\quad \forall (i, j) \in E,
			\label{stabelset:char2}\\
		\card{\set{i \in B \mid t \in S_i}} &\leq \frac{\card{V} - 1}{2}
			\quad \forall t \in U
			\label{stabelset:char3}
	\end{alignat}
	(\cf \cref{tb:set-char}).
	For a given point $ h \in H(G) $, we then need to find
	sets $ S_v $ for each $ v \in V $
	such that $ \mu(S_v) = h_v $ and the above conditions hold.
	We define these sets in routine \textsc{Define-Odd-Cycle-Stable-Sets-Subsets},
	given in \cref{fig:stable-set-odd-cycle}.
	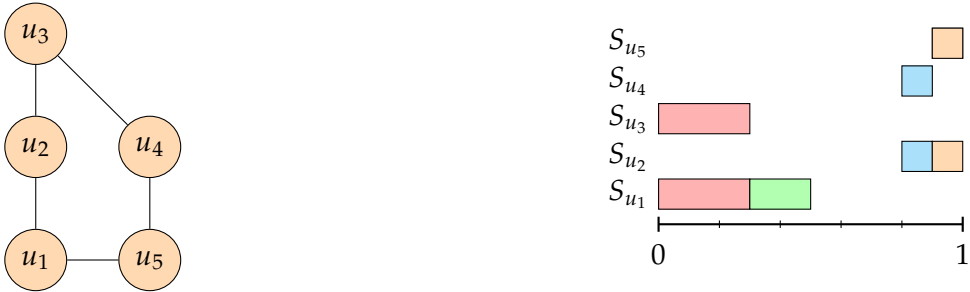
\begin{figure}
		\centering
		\begin{subfigure}{\textwidth}
			\begin{algorithmic}[1]
				\Function{Define-odd-cycle-Stable-Set-Subsets}{}
					\State Let $ (v_1, \ldots, v_{\card{V}}) $ be any ordering
						of the elements in~$V$\\
						\qquad\qquad
						with $ \set{v_i, v_{i + 1}} \in E $
						for all $ i \in [\card{V} - 1] $
					\State $ \bar{h} \leftarrow h $
					\For{each $ v \in (v_1, \ldots v_{\card{V}}) $}
						\Comment{Blow-up phase}
						\State $ r \leftarrow \frac{\card{V} - 1}{2}
							- \sum_{i = 1}^{\card{V}} \bar{h}_i $
						\State $ \bar{h}_v \leftarrow
							\min(1 - \bar{h}_{v - 1},
								1 - \bar{h}_{v + 1},
								\bar{h}_v + r) $
							\Comment{Make~\cref{stab:set-inq}
								or~\cref{ineq:cycle} active}
					\EndFor
					\State $ (\bar{S}_{v_1}, \ldots, \bar{S}_{v_{\card{V}}})
						\leftarrow \textsc{Match}([0, 1),
							(\bar{h}_{v_1}, \ldots, \bar{h}_{v_{\card{V}}}) $
						\Comment{Set auxiliary sets with buffers}
					\For{each $ v \in (v_1, \ldots v_{\card{V}} $)}
						\Comment{Reduction phase}
						\State $ (S_v) \coloneqq \textsc{Match}(\bar{S}_v, (h_v)) $
					\EndFor
				\EndFunction
			\end{algorithmic}
		\end{subfigure}
		\par\vskip\floatsep
		\begin{subfigure}{.5\textwidth}
			\hspace{1cm}
			\begin{tikzpicture}[xscale=1.5,yscale=1.5]
				\node[shape=circle,draw=black, fill=orange!30] (A) at (0,0) {$u_1$};
				\node[shape=circle,draw=black, fill=orange!30] (B) at (0,1) {$u_2$};
				\node[shape=circle,draw=black, fill=orange!30] (C) at (0,2) {$u_3$};
				\node[shape=circle,draw=black, fill=orange!30] (E) at (1,0) {$u_5$};
				\node[shape=circle,draw=black, fill=orange!30] (D) at (1,1) {$u_4$};
				\path [-] (A) edge node[above] {} (B);
				\path [-] (B) edge node[above] {} (C);
				\path [-] (C) edge node[above] {} (D);
				\path [-] (D) edge node[above] {} (E);
				\path [-] (E) edge node[above] {} (A);
			\end{tikzpicture}
		\end{subfigure}
		\hfil
		\begin{subfigure}{.45\textwidth}
			\centering
			\hspace{-1cm}
			\begin{tikzpicture}[xscale=4,yscale=1]
				\draw[thick] (0,0) -- (1,0);
				\draw[thick] (0,.1) -- (0,.-.1);
				\draw (.2,.05) -- (.2,-.05);
				\draw (.4,.05) -- (.4,-.05);
				\draw (.6,.05) -- (.6,-.05);
				\draw (.8,.05) -- (.8,-.05);
				\draw[thick] (1,.1) -- (1,.-.1);
				\node at (0,-.4) {$0$};
				\node at (1,-.4) {$1$};
				
				\newcommand\AONE{.5*0}
				\newcommand\ATWO{.5*1}
				\newcommand\ATHREE{.5*2}
				\newcommand\AFOUR{.5*3}
				\newcommand\AFIVE{.5*4}
				
				\newcommand\BLOCKHEIGHT{0.4}
				\draw[fill=red!30] (0,0.2+\AONE) rectangle ++(0.3,\BLOCKHEIGHT);
				\draw[fill=red!30] (0,0.2+\ATHREE) rectangle ++(0.3,\BLOCKHEIGHT);
				
				\draw[fill=green!30] (0.3,0.2+\AONE) rectangle ++(0.2,\BLOCKHEIGHT);
				
				\draw[fill=cyan!30] (0.8,0.2+\ATWO) rectangle ++(0.1,\BLOCKHEIGHT);		
				\draw[fill=cyan!30] (0.8,0.2+\AFOUR) rectangle ++(0.1,\BLOCKHEIGHT);				
				\draw[fill=orange!30] (0.9,0.2+\AFIVE) rectangle ++(0.1,\BLOCKHEIGHT);
				
				\draw[fill=orange!30] (0.9,0.2+\ATWO) rectangle ++(0.1,\BLOCKHEIGHT);
				\draw[fill=orange!30] (0.9,0.2+\AFIVE) rectangle ++(0.1,\BLOCKHEIGHT);
				
				\node at (-.1,.4+\AONE) {$S_{u_1}$};
				\node at (-.1,.4+\ATWO) {$S_{u_2}$};
				\node at (-.1,.4+\ATHREE) {$S_{u_3}$};
				\node at (-.1,.4+\AFOUR) {$S_{u_4}$};
				\node at (-.1,.4+\AFIVE) {$S_{u_5}$};
			\end{tikzpicture}
		\end{subfigure}
		\caption{Routine \textsc{Define-odd-cycle-Stable-Set-Subsets} (top),
			exemplary graph with five nodes (bottom left)
			and possible output of the routine
			for the point given by $ h = (0.5, 0.2, 0.3, 0.1, 0.1) $.
			It is blown up to $ \bar{h} = (0.8, 0.2, 0.8, 0.1, 0.1) $.
			The point~$h$ can be written as a convex combination
			of incidence vectors belonging to five stable sets,
			namely $ \set{u_1, u_3} $, $ \set{u_1} $, $ \emptyset $,
			$ \set{u_2, u_4} $ and $\set{u_2, u_5} $,
			each non-empty one marked with same colour (bottom right).
		}
		\label{fig:stable-set-odd-cycle}
		\null\vfill\null
	\end{figure}
	First, in Line~2/3, we fix an ordering of the nodes
	which respects the order of the cycle.
	In Lines~4--8, the point~$h$ is then shifted to a point~$ \bar{h} $
	on the boundary of~$ H(G) $ by increasing~$h$ componentwise
	until in each iteration
	at least one of the inequalities~\cref{stab:set-inq,ineq:cycle} becomes active.
	By induction, for the resulting point $ \bar{h} $
	the inequality $ \bar{h} \geq h $ holds component-wise
	and we have $ \sum_{i = 1}^{\card{V}} \bar{h}_i = (\card{V} - 1)/2 $.
	Now, auxiliary sets $ \bar{S}_v $, $ v \in V $,
	are placed in consecutive order along the cycle in Line~9,
	based on the diameters stored in $ \bar{h} $.
	Observe that, in particular, the first set
	is defined as $ \bar{S}_{v_1} = [0, \bar{h}_{v_1}) $
	and the last one as $ \bar{S}_{v_{\card{V}}}
		= [1 - \bar{h}_{v_{\card{V}}}, 1) $,
	thus they satisfy set characterizations~\cref{stabelset:char2,stabelset:char3}.
	Finally, in Lines~10--12, the diameters of the auxiliary sets are reduced
	such that they correspond to the components of~$h$
	to obtain the final sets $ S_v $, $ v \in V $.
	It is obvious that these sets satisfy $ \mu(S_v) = h_v $ for all $ v \in V $,
	and the reduction does not invalidate
	any of the set characterizations~\cref{stabelset:char2}
	or~\cref{stabelset:char3}.
	Therefore, we have proved $ H(G) \subseteq P(G) $.
\end{proof}
In the above proof, an auxiliary point $ \bar{h} \in H $ is constructed
by greedily increasing the coordinates of the point~$h$ to be tested.
The sets for $ \bar{h} $ are then placed next to each other,
modulo $1$ (the diameter of $U$).
The backward transformation then simply shrinks the sets
to fit the size of the original coordinates of~$h$
while maintaining the validity of all set characterizations.
As shown in \cref{fig:stable-set-odd-cycle}, the final sets after backward transformation
are not always placed next to each other due to the gaps arising from the shrinking step.
A direct placement of these sets for the original point seems to more involved,
since it is not obvious how to calculate the gaps between adjacent sets a priori.

\section{Extensions of Zuckerberg's method for general convex sets}
\label{sec:zuckerberg_extended}

Both the original proof technique by Zuckerberg
from \cite{zuckerberg2016geometric}
and its simplification in \cite{gupte2020extended}
are applicable to $0$/$1$-polytopes only.
In the following, we will derive extensions of Zuckerberg's method
which enable us to conduct geometric convex-hull proofs
for arbitrary convex sets.
This includes, in particular, general integer polyhedra.
The underlying idea is to pass from intervals in $ U = [0, 1) $
to rectangles in $ \R^2 $
when constructing the sets to represent a given point~$h$
in some convex set~$H$.
Recall that the original method interprets each of these dimension-many sets
as either a~$0$- or a~$1$-coordinate of a vertex in a $0$/$1$-polytope;
a coordinate of the vertex which belongs to some $ t \in U $ is~$1$
if the corresponding set includes~$t$, and~$0$ otherwise.
The vertices associated with the sets representing a point~$h$ in the polytope
define a convex combination spanning~$h$.
Our extension of Zuckerberg's method gives the intervals
making up these sets a height
to encode the coordinates of arbitrary points in~$H$
instead of only $0$/$1$-points.
This idea will lead to generalized versions
of the theorems in \cref{Sec:Geometric_Proofs}
which can be used to prove the completeness
of convex-hull representations for general convex sets.
Furthermore, they also allow to compute convex combinations
spanning a certain~$ h \in H $ using any points in~$H$,
not necessarily extreme points.

To formalize the new approach, we first define the set 
\begin{equation*}
	\RR^Q \coloneqq \SSet{([a, b), c) \in \mP(Q) \times \R}{a < b,\, c \neq 0}, 
\end{equation*}
where~$Q$ is chosen as either $ U = [0, 1) $ or as $ \R_+ $.
The set~$Q$ specifies the range of coefficients which are allowed
in a linear combination representing some $ h \in H $.
We use $ Q = U $ to construct convex combinations
and $ Q = \R_+ $ for conic combinations.
We interpret $ \RR^Q $ as the set of all non-degenerate,
axis-parallel rectangles~$R$ in~$ \R^2 $,
which are uniquely defined by stating
the two diagonally opposite vertices~$ (a, 0) $ and~$ (b, c) $.
The sign of~$c$ indicates if a rectangle
points into the upper half-space ($ c > 0 $)
or the lower half-space ($ c < 0 $).
Let $ q(R) \coloneqq (b - a) c $ and $ z(R) \coloneqq c $
denote the signed area and the signed height of the rectangle~$R$ respectively.
Further, let $ y\colon Q \times \RR^Q \to \F $ be an indicator function
defined as follows.
For some $ t \in Q $ and $ R \in \RR^Q $
it is $ y(t, R) = 1 $ if $ a \leq t \leq b $, and $ y(t, R) = 0 $ otherwise.
In other words, $y$~indicates whether~$t$ belongs to the support of~$R$,
in which case we call~$R$ \emph{active} at $t$.
We call two rectangles~$ R_1 $ and~$ R_2 $ \emph{non-overlapping}
if there exists no $ t \in Q $
such that both $ y(t, R_1) = 1 $ and $ y(t, R_2) = 1 $ hold.

In a similar fashion as in \cref{Sec:Geometric_Proofs},
we then define $ \LLE{Q} $ as the set of all unions of finitely many
non-degenerate, non-overlapping rectangles from $ \RR^Q $ 
and $ \muE{Q} $ as the Lebesgue measure restricted to $ \LLE{Q} $, that is
\begin{equation*}
		\LLE{Q} \coloneqq \SSet{\set{R_1,\ldots,R_k}}{
		\begin{split}
			&k \in \N \wedge R_1, \ldots, R_k \in \RR^Q \wedge \\
			&R_i \text{ and } R_j \text{ are non-overlapping}
				\quad \forall i, j \in [k], i \neq j
	\end{split}}
\end{equation*}
\begin{equation*}
		\muE{Q}(S) \coloneqq \sum_{i = 1}^k q(R_i)
			\text{ for any } S = \set{R_1, \ldots, R_k} \in \LLE{Q}.
\end{equation*}
Moreover, we define the indicator function
$ \func{\phiE{Q}}{Q \times \LLE{Q}}{\R} $,
\begin{equation*}
	\phiE{Q}(t, S) \coloneqq \begin{cases}
		z(R) & \text{if } y(t, R) = 1
			\text{ for any } R \in \set{R_1, \ldots, R_k},\\
		0 & \text{otherwise},
	\end{cases}
\end{equation*}
where~$S$ is uniquely represented as $ S = \set{R_1,\ldots,R_k} $
for some $ k \in \N $ in $ R_i \in \RR^Q $, $ i \in [k] $.
It returns the height of the rectangle which is active at $ t \in Q $
if there is one.
Note that the active rectangle is unique in this case
as the $ R_i $ forming $S$ are non-overlapping.
Finally, let $ \func{\varphiE{Q}}{Q \times (\LLE{Q})^n}{\R^n},
	\varphiE{Q}(t, S_1, \ldots, S_n) = v $,
where $ v_i \coloneqq \phiE{Q}(t, S_i) $ for $ i \in [n] $.
Here we interpret the heights of the rectangles which are active at~$t$
as the coordinates of a vector in $ \R^n $.

With the above definitions, we are equipped to state
our extensions of Zuckerberg's method.
As a useful shorthand notation used in the proofs,
we define the union $ S_1 \cup S_2 $ of two sets $ S_1, S_2 \in \LLE{Q} $
as the unique $ S \in \LLE{Q} $ such that for all $ t \in Q $
we have $ \phiE{Q}(t, S) = \phiE{Q}(t, S_1) + \phiE{Q}(t, S_2) $.
Informally speaking, this means we add the heights of the rectangles
which are active at a certain~$t$ to form the union~$S$.

We start with an extension which enables us
to conduct geometric convex-hull proofs for general convex sets.
\begin{theorem}[Zuckerberg's method for general convex sets]
	\label{thm:extended_zucker}
	Let $ \FF \subseteq \R^n $ and $ h \in \R^n $.
	Then we have $ h \in \conv(\FF) $
	iff there are sets $ S_1, \ldots, S_n \in \LLE{U} $
	such that $ \muE{U}(S_i) = h_i $ for all $ i \in [n] $
	and $ \varphiE{U}(t, S_1, \ldots, S_n) \in \conv(\FF) $
	for every $ t \in U $.
\end{theorem}
\begin{proof}
	If $ h \in \conv(\FF) $,
	then there exist $ \xi^1, \ldots, \xi^r \in \conv(\FF) $,
	for some $ r \in \N_+ $, such that~$h$ can be written as
	$ h = \lambda_1 \xi^1 + \ldots + \lambda_r \xi^r $
	with $ \lambda_1 + \ldots + \lambda_r = 1 $
	and $ \lambda_k \geq 0 $ for all $ k \in [r] $.
	We can then define a partition $ U = I_1 \cup \ldots \cup I_r $
	by setting $ I_1 \coloneqq [0, \lambda_1) $
	and $ I_k \coloneqq \left[\lambda_1 + \ldots + \lambda_{k - 1},
		\lambda_1 + \ldots + \lambda_k\right) $
	for $ k \in \set{2, \ldots, r} $.
	This allows us to set
	\begin{equation*}
		S_i \coloneqq \bigcup_{k \in [r]:\, \xi^k_i \neq 0}(I_k, \xi^k_i)
			\in \LLE{U},
			\quad i \in [n],
	\end{equation*}
	with rectangles $ (I_k, \xi^k_i) \in \RR^U $.
	For all $ i \in [n] $, we can conclude
	\begin{equation*}
		\muE{U}(S_i) = \sum_{k \in [r]:\, \xi^k_i \neq 0} \muE{U}((I_k, \xi^k_i))
			= \sum_{k \in [r]:\, \xi^k_i \neq 0} \lambda_k \xi^k_i
			= \sum_{k \in [r]} \lambda_k \xi^k_i = h_i.
	\end{equation*}
	Furthermore, for every $ t \in U $
	there is a unique index~$k$ with $ t \in I_k $,
	and thus we have
	\begin{equation*}
		\varphiE{U}(t, S_1, \ldots, S_n) = \xi^k \in \conv(\FF).
	\end{equation*}
	
	Conversely, if the $ S_i $ are sets with the described properties,
	let $ \xi^1, \ldots, \xi^r $ be an ordering of the elements in
	\begin{equation*}
		\SSet{\xi \in \conv(\FF)}{\varphiE{U}(t, S_1, \ldots, S_n) = \xi
			\text{ for some } t \in U}.
	\end{equation*}
	The above set is finite, since each $ S_i $ is a finite union of rectangles. 
	We can set, by slight abuse of notation,
	\begin{equation*}
		\lambda_k \coloneqq \muE{U}\left(\SSet{(t, 1) \in U \times \set{1}}
			{\varphiE{U}(t, S_1, \ldots, S_n) = \xi^k}\right)
	\end{equation*}
	for $ k \in [r] $ to obtain the required convex representation
	$ h = \lambda_1 \xi^1 + \ldots + \lambda_r \xi^r $.
	To see this, we can easily verify
	\begin{equation*}
		S_i = \bigcup_{k \in [r]:\, \xi^k_i \neq 0}
			\SSet{(t, \xi_i^k) \in U \times \R}
				{\varphiE{U}(t, S_1, \ldots, S_n) = \xi^k}
	\end{equation*}
	for all $ i \in [n] $.
	We then conclude for all $ i \in [n] $:
	\begin{alignat*}{1}
		\sum_{k \in [r]} \lambda_k \xi^k_i
			&= \sum_{k \in [r]} \muE{U} \left(\SSet{(t, 1) \in U \times \set{1}}
				{\varphiE{U}(t, S_1, \ldots, S_n) = \xi^k}\right) \xi^k_i\\
			&= \sum_{k \in [r]:\, \xi^k_i \neq 0}
				\muE{U} \left(\SSet{(t, \xi^k_i) \in U \times \R}
					{\varphiE{U}(t, S_1, \ldots, S_n) = \xi^k}\right)\\
			&= \muE{U}\left(\bigcup_{k \in [r]:\, \xi^k_i \neq 0}
				\SSet{(t, \xi^k_i) \in U \times \R}
					{\varphiE{U}(t, S_1, \ldots, S_n) = \xi^k}\right)\\
			& = \muE{U}(S_i) = h_i.\\[-2\baselineskip]
	\end{alignat*}
	\qed
\end{proof}
\cref{thm:extended_zucker} generalizes \cref{thm:gupteC} in two ways.
The method now works for arbitrary convex sets,
instead of only $0$/$1$-polytopes.
Note that Zuckerberg's original method can be recovered
by discarding the height of the rectangles
and only checking whether a given set is active at some $ t \in U $.
We also remark that in \cref{thm:extended_zucker}
we can now write the point~$h$
as a linear combination of points in $ \conv(\FF) $,
not only points in $ \FF $.
This allows an additional degree of freedom for convex-hull proofs
which \cref{thm:gupteC} does not offer.

Using our extended framework, we can also determine a representation
of any given point $ h \in H $
as a convex combination of points in $ \FF $
if we find corresponding sets $ S_1, \ldots, S_n $
fulfilling the requirements of \cref{thm:extended_zucker}.
To state this result, we define the two sets
\begin{equation*}
	\bar{\FF}(S_1, \ldots, S_n) \coloneqq
		\SSet{\xi \in \conv(\FF)}{\varphiE{U}(t, S_1, \ldots, S_n) = \xi
			\text{ for some } t \in U}
\end{equation*}
and, for each $ \xi \in \R^n $,
\begin{equation*}
	\LE{U}_{\xi}(S_1,\dotsc,S_n) \coloneqq \SSet{t \in U}
		{\varphiE{U}(t, S_1, \ldots, S_n) = \xi}.
\end{equation*}
The following corollary then directly follows
from the proof of \cref{thm:extended_zucker}.
\begin{corollary}[Convex combinations for general convex sets]
	Under the same assumptions as in \cref{thm:extended_zucker},
	let $ \lambda_\xi \coloneqq \muE{U}(\LE{U}_{\xi}(S_1, \ldots, S_n)) $
	for each $ \xi \in \bar{\FF}(S_1, \ldots, S_n) $.
	Then we have $ h = \sum_{\xi \in \bar{\FF}(S_1, \ldots, S_n)} \lambda_\xi \xi $,
	$ \sum_{\xi \in \bar{\FF}} \lambda_\xi = 1 $
	and $ \lambda_\xi \geq 0 $ for all $ \xi \in \bar{\FF}(S_1, \ldots, S_n) $.
\end{corollary}
The definition of a set characterization (\cref{def:set-char})
can now be restated in a more general form as well.
\begin{definition}[Set characterization of a constraint]
	\label{def:set-char2}
	Let $ \func{f}{\FF}{\R} $, let $ b \in \R $,
	and let $ S_1, \ldots S_n \in \LLE{U} $.
	The \emph{set characterization} of some constraint $ f(x) \leq b $
	is the following logic statement:
	\begin{equation*}
		f(\phiE{U}(t, S_1), \ldots, \phiE{U}(t, S_n)) \leq b \text{ holds for all } t \in U.
	\end{equation*}
\end{definition}
Similar to before, we can use this concept to facilitate finding sets $ S_i $
which characterize a point $ h \in H $ according to \cref{thm:extended_zucker}.
\begin{lemma}
	\label{thm:extended_working_thm}
	Let $ \FF \coloneqq \set{x \in \Z^n \mid f_j(x) \leq b_j\, \forall j \in [m]} $
	for some $ m \in \N $.
	Further, let $ P \coloneqq \conv(\FF) $,
	and let~$ H \subseteq \R^n $ be some convex set.
	We have $ H = P $ iff both $ \FF \subseteq H $ holds
	and for each $ h \in H $ there are sets $ S_1, \ldots, S_n \in \LLE{U} $
	with $ \muE{U}(S_i) = h_i $ for all $ i \in [n] $
	which satisfy the set characterization
	for each constraint $ f_j(x) \leq b_j $, $ j \in [m] $.
\end{lemma}

Polyhedra, which are special convex sets,
can be written as a convex combination of a finite sets of points
plus a conic combination of a finite set of rays
In the following, we give an alternative version of \cref{thm:extended_zucker}
for polyhedra making use of this fact.
\begin{theorem}[Zuckerberg's method for polyhedra]
	\label{thm:extended_zucker2}
	Let $ \FF \subseteq \R^n $ and $ \EE \subseteq \R^n $
	be a finite, non-empty set of points, $ h \in \R^n $.
	Then we have $ h \in \conv(\FF) + \cone(\EE) $
	iff there are sets $ S_1, \ldots, S_n \in \LLE{U} $
	and sets $ S'_1, \ldots,S'_n \in \LLE{\R_+} $
	such that $ \muE{U}(S_i) + \muE{\R_+}(S'_i) = h_i $ for all $ i \in [n] $,
	$ \varphiE{U}(t, S_1, \ldots, S_n) \in \conv(\FF) $ for all $ t \in U $
	and $ \varphiE{\R_+}(t', S'_1, \ldots, S'_n) \in \cone(\EE) $
	for all $ t' \in \R_+ $.
\end{theorem}
\begin{proof}
	If $ h \in \conv(\FF) + \cone(\EE) $,
	then there exist $ \xi^1, \ldots, \xi^r \in \conv(\FF) $
	for some $ r \in \N_+ $ and $ \zeta^1, \ldots, \zeta^q \in \cone(\EE) $
	for some $ q \in \N_+ $ such that~$h$ can be written as
	$ h = \lambda_1 \xi^1 + \ldots + \lambda_r \xi^r
		+ \eta_1 \zeta^1 + \ldots + \eta_q \zeta^q $
	with $ \lambda_1 + \ldots + \lambda_r = 1 $,
	$ \lambda_k \geq 0 $ for all $ k \in[r] $
	and $ \eta_k \geq 0 $ for all $ k \in [q] $.
	Define now the partition $ U = I_1 \cup \ldots \cup I_r $
	by setting $ I_1 \coloneqq [0, \lambda_1) $
	and $ I_k \coloneqq \left[\lambda_1 + \ldots + \lambda_{k - 1},
		\lambda_1 + \ldots + \lambda_k\right) $
	for $ k \in \set{2, \ldots, r} $.
	In addition, we define $ \bar{I}_1 \coloneqq [0, \eta_1) $
	and $ \bar{I}_k \coloneqq \left[\eta_1 + \ldots + \eta_{k - 1},
		\eta_1 + \ldots + \eta_k\right) $
	for $ k \in \set{2, \ldots, q} $.
	With
	\begin{equation*}
		S_i \coloneqq \bigcup_{k \in [r]:\, \xi^k_i \neq 0}(I_k, \xi^k_i),
			\quad S'_i \coloneqq \bigcup_{k \in [q]:\, \zeta^k_i \neq 0}
				(\bar{I}_k, \zeta^k_i)
	\end{equation*}
	for $ i \in [n] $, we find
	\begin{alignat*}{1}
		\muE{U}(S_i) + \muE{\R_+}(S'_i)
			&= \sum_{k \in [r]:\, \xi^k_i \neq 0} \muE{U}((I_k,\xi^k_i))
				+ \sum_{k \in [q]:\, \xi^k_i \neq 0}
					\muE{\R_+}((\bar{I}_k, \zeta^k_i))\\
			&= \sum_{k \in [r]:\, \xi^k_i \neq 0} \lambda_k \xi^k_i
				+ \sum_{k \in [q]:\, \zeta^k_i \neq 0} \eta_k \zeta^k_i
			= \sum_{k \in [r]} \lambda_k \xi^k_i +
				\sum_{k \in [q]} \eta_k \zeta^k_i\\
			&= h_i.
	\end{alignat*}
	Moreover, for each $ t \in U $, there is a unique index~$k$ with $ t \in I_k $,
	and thus
	\begin{equation*}
		\varphiE{U}(t, S_1, \ldots, S_n) = \xi^k \in \conv(\FF).
	\end{equation*}
	Similarly, for each $ t \in \R_+ $,
	there is either a unique index~$k$ with $ t \in \bar{I}_k $,
	or there is no such index.
	Thus, we conclude
	\begin{equation*}
		\varphiE{\R_+}(t, S'_1, \ldots, S'_n) = \zeta \in \cone(\EE).
	\end{equation*}
	Especially, if there is no index~$k$ as outline above,
	we have $ \zeta = 0 $.
	
	Conversely, if~$ S_i $ and~$ S'_i $ are sets with the stated properties, 
	let $ \xi^1, \ldots \xi^r $ be an ordering of the elements in
	\begin{equation*}
		\SSet{\xi \in \conv(\FF)}
			{\varphiE{U}(t, S_1, \ldots, S_n) = \xi
				\text{ for some } t \in U},
	\end{equation*}
	and let $ \zeta_1, \ldots \zeta_q $ be an ordering of the elements in
	\begin{equation*}
		\SSet{\zeta \in \cone(\EE)}
			{\varphiE{\R_+}(t, S'_1, \ldots, S'_n) = \zeta
				\text{ for some } t \in \R_+}.
	\end{equation*}
	Both sets are finite, since all $ S_i $ and $ S'_i $
	are finite unions of rectangles.  
	For $ k \in [r] $ and $ p \in [q] $, we can set
	\begin{alignat*}{1}
		\lambda_k &\coloneqq \muE{U} \left(\SSet{(t, 1) \in U \times \set{1}}
			{\varphiE{U}(t, S_1, \ldots, S_n) = \xi^k}\right),\\
		\eta_p &\coloneqq \muE{\R_+} \left(\SSet{(t, 1) \in \R_+ \times \set{1}}
			{\varphiE{\R_+}(t, S'_1, \ldots, S'_n) = \zeta^p}\right)
	\end{alignat*}
	to obtain the required convex representation
	$ h = \lambda_1 \xi^1 + \ldots + \lambda_r \xi^r
		+ \eta_1 \zeta^1 + \ldots + \eta_q \zeta^q $. 
	To this end, observe that for all $i \in [n] $ we have
	\begin{alignat*}{1}
		S_i &= \bigcup_{k \in [r]:\, \xi^k_i \neq 0}
			\SSet{(t, \xi_i^k) \in U \times \R}
				{\varphiE{U}(t, S_1, \ldots, S_n) = \xi^k},\\
		S'_i &= \bigcup_{k \in [q]:\, \zeta^k_i \neq 0}
			\SSet{(t, \xi_i^k) \in \R_+ \times \R}
				{\varphiE{\R_+}(t, S'_1, \ldots, S'_n) = \zeta^k}.
	\end{alignat*}
	For all $ i \in [n]  $, this leads to
	\begin{alignat*}{1}
		\sum_{k \in [r]} \lambda_k \xi^k_i
			&= \sum_{k \in [r]} \muE{U}
				\left(\SSet{(t, 1) \in U \times \set{1}}
					{\varphiE{U}(t, S_1, \ldots, S_n) = \xi^k}\right) \xi^k_i\\
			&= \sum_{k \in [r]:\, \xi^k_i \neq 0} \muE{U}
				\left(\SSet{(t, \xi^k_i) \in U \times \R}
					{\varphiE{U}(t, S_1, \ldots, S_n) = \xi^k}\right)\\
			&= \muE{U} \left(\bigcup_{k \in [r]:\, \xi^k_i \neq 0}
				\SSet{(t, \xi^k_i) \in U \times \R}
					{\varphiE{U}(t, S_1, \ldots, S_n) = \xi^k}\right)\\
			&= \muE{U}(S_i)
	\end{alignat*}
	and
	\begin{alignat*}{1}
		\sum_{k \in [q]} \eta_k \zeta^k_i
			&= \sum_{k \in [q]} \muE{\R_+}
				\left(\SSet{(t, 1) \in \R_+ \times \set{1}}
					{\varphiE{\R_+}(t, S'_1, \ldots, S'_n)
						= \zeta^k}\right) \zeta^k_i\\
			&= \sum_{k \in [q]:\, \zeta^k_i \neq 0} \muE{\R_+}
				\left(\SSet{(t, \zeta^k_i) \in \R_+ \times \R}
					{\varphiE{\R_+}(t, S'_1, \ldots, S'_n) = \zeta^k}\right)\\
			&= \muE{\R_+} \left(\bigcup_{k \in [p]:\, \zeta^k_i \neq 0}
				\SSet{(t, \zeta^k_i) \in \R_+ \times \R}
					{\varphiE{\R_+}(t, S'_1, \ldots, S'_n) = \zeta^k}\right)\\
			&= \muE{\R_+}(S'_i).
	\end{alignat*}
	This yields
	\begin{align*}
		\sum_{k = 1}^r \lambda_k \xi^k_i + \sum_{k = 1}^q \eta_k \zeta^k_i
			= \muE{U}(S_i) + \muE{\R_+}(S'_i) = h_i.\\[-2\baselineskip]
	\end{align*}
	\qed
\end{proof}
When we conduct a convex-hull proof via \cref{thm:extended_zucker},
we implicitly write the given point~$h$
as a convex combination of other points in~$H$
(most often extreme points).
In contrast, \cref{thm:extended_zucker2} allows us to express~$h$
as both a convex and conic combination of points spanning~$H$.
This is especially interesting for polyhedra,
which can be split into a convex and a conic part.
Both versions are valuable tools and allow for different proof strategies
as the example in \cref{simplex:unb} shows.

If we succeed in giving a convex-hull proof hull via \cref{thm:extended_zucker2},
we can again deduce convex and conic combinations afterwards.
\begin{corollary}[Convex combinations for polyhedra]
	Under the same assumptions as in \cref{thm:extended_zucker2},
	let $ \lambda_\xi \coloneqq \muE{U}(\LE{U}_{\xi}(S_1, \ldots, S_n)) $
	for all $ \xi \in \FF $
	and $ \eta_\zeta \coloneqq \muE{\R_+}(\LE{\R_+}_{\zeta}(S'_1, \ldots, S'_n)) $
	for all $ \zeta \in \EE $.
	Then we have $ h = \sum_{\xi \in \FF} \lambda_\xi \xi
		+ \sum_{\zeta \in \EE} \eta_\zeta \zeta $,
	$ \sum_{\xi \in \FF} \lambda_\xi = 1 $
	with $ \lambda_\xi \geq 0 $ for all $ \xi \in \FF $
	and $ \eta_\zeta \geq 0 $ for all $ \zeta \in \EE $.
\end{corollary}
It is straightforward to adjust the definition of set characterizations
from \cref{def:set-char2} to include the conic part as well.
We will, however, skip this for reasons of space.

To define the sets in a convex-hull proof according to \cref{thm:extended_zucker}
or \cref{thm:extended_zucker2}, it will be helpful
to introduce the auxiliary function $ o\colon Q \times U \to \LL \times U $,
\begin{equation*}
	o(t, a) \coloneqq \begin{cases}
		([t, t + a), t + a) & \text{if } t + a \leq 1,\\
		([t, 1) \cup [0, t + a - 1), t + a - 1) & \text{otherwise}.
		\end{cases}
\end{equation*}
This function determines an interval starting at $ t \in Q $
and of diameter~$ a \in U $, modulo~$1$.
It returns an ordered pair consisting of the interval and its end point,
which will be useful when placing rectangles adjacent to each other.

We will now give some indicative first examples to illustrate
how the results derived in this section can be used to give convex-hull proofs.

\subsection{Convex-hull proofs using interior points}

We start with the example of a simplex
to show that the point $ h \in H $ does not necessarily have to be written
as a convex combination of vertices, but that it is also possible
to characterize it via sets corresponding to other points in the interior.
Let $ \FF \coloneqq \set{x \in \Z^n \mid \sum_{i = 1}^n x_i \leq b,\, x \geq 0} $
and $ H \coloneqq \set{x \in \R^n \mid \sum_{i = 1}^n x_i \leq b,\, x \geq 0} $
with some $ b \in \N $.
The set characterizations for the simplex constraint
and the non-negativity constraint can be stated as
\begin{alignat}{1}
	\sum_{i = 1}^n \phiE{U}(t, S_i) &\leq b \quad \forall t \in U,\label{char:simplex1}\\
	\phiE{U}(t, S_i) &\geq 0 \quad \forall i \in [n],\, \forall t \in U\label{char:simplex2}
\end{alignat}
respectively.
A possible construction of the sets~$ S_i $ for \cref{thm:extended_zucker2}
is given in routine \textsc{Define-Simplex-Subsets-A},
and a different variant is given in routine \textsc{Define-Simplex-Subsets-B},
both stated in \cref{fig:simplex}.
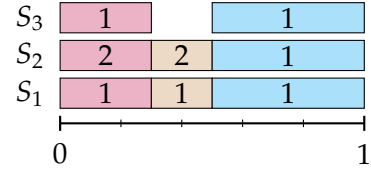
\begin{figure}[h]
	\centering
	\begin{subfigure}{.5\textwidth}
		\begin{algorithmic}[1]
			\Function{Define-Simplex-Subsets-A}{}
			\State $ r \leftarrow 0 $
			\For{$ i \in [n] $}
				\State $ S_i \coloneqq ([r, r + h_i / b), b) $
				\State $ r \leftarrow r + h_i / b $
			\EndFor
			\EndFunction
		\end{algorithmic}
	\end{subfigure}
	\hfil
	\begin{subfigure}{.45\textwidth}
		\centering
		\begin{tikzpicture}[xscale=4,yscale=1]
			\draw[thick] (0,0) -- (1,0);
			\draw[thick] (0,.1) -- (0,.-.1);
			\draw (.2,.05) -- (.2,-.05);
			\draw (.4,.05) -- (.4,-.05);
			\draw (.6,.05) -- (.6,-.05);
			\draw (.8,.05) -- (.8,-.05);
			\draw[thick] (1,.1) -- (1,.-.1);
			\node at (0,-.4) {$0$};
			\node at (1,-.4) {$1$};
			\draw[fill=red!30] (0,0.2) rectangle ++(0.25,0.4) node[pos=.5] {$4$};
			\draw[fill=blue!30] (0.25,0.2+.5*1) rectangle ++(0.375,0.4) node[pos=.5] {$4$};
			\draw[fill=green!30] (0.625,0.2+.5*2) rectangle ++(0.2,0.4) node[pos=.5] {$4$};
			
			\node at (-.1,.4+.5*0) {$\S_1$};
			\node at (-.1,.4+.5*1) {$\S_2$};
			\node at (-.1,.4+.5*2) {$\S_3$};
		\end{tikzpicture}
	\end{subfigure}
	\par\vskip\floatsep
	\begin{subfigure}{.5\textwidth}
		\begin{algorithmic}[1]
			\Function{Define-Simplex-Subsets-B}{}
				\State $ r \leftarrow 0 $
				\For{$ i \in [n] $}
					\State $ (I, r) \leftarrow o(r, h_i - \myfloor{h_i}) $
					\State $ S_i \coloneqq ([0, 1), \myfloor{h_i}) \cup (I, 1)) $
				\EndFor
			\EndFunction
		\end{algorithmic}
	\end{subfigure}
	\hfil
	\begin{subfigure}{.45\textwidth}
		\centering
		\begin{tikzpicture}[xscale=4,yscale=1]
			\draw[thick] (0,0) -- (1,0);
			\draw[thick] (0,.1) -- (0,.-.1);
			\draw (.2,.05) -- (.2,-.05);
			\draw (.4,.05) -- (.4,-.05);
			\draw (.6,.05) -- (.6,-.05);
			\draw (.8,.05) -- (.8,-.05);
			\draw[thick] (1,.1) -- (1,.-.1);
			\node at (0,-.4) {$0$};
			\node at (1,-.4) {$1$};
			\draw[fill=purple!30] (0,0.2) rectangle ++(0.3,0.4) node[pos=.5] {$1$};
			
			\draw[fill=purple!30] (0,0.2+.5*1) rectangle ++(0.3,0.4) node[pos=.5] {$2$};
			
			\draw[fill=purple!30] (0,0.2+.5*2) rectangle ++(0.3,0.4) node[pos=.5] {$1$};
			
			\draw[fill=brown!30] (0.3,0.2+.5*0) rectangle ++(0.2,0.4) node[pos=.5] {$1$};
			\draw[fill=brown!30] (0.3,0.2+.5*1) rectangle ++(0.2,0.4) node[pos=.5] {$2$};
			
			\draw[fill=cyan!30] (0.5,0.2+.5*0) rectangle ++(0.5,0.4) node[pos=.5] {$1$};
			\draw[fill=cyan!30] (0.5,0.2+.5*1) rectangle ++(0.5,0.4) node[pos=.5] {$1$};
			\draw[fill=cyan!30] (0.5,0.2+.5*2) rectangle ++(0.5,0.4) node[pos=.5] {$1$};
			
			\node at (-.1,.4+.5*0) {$\S_1$};
			\node at (-.1,.4+.5*1) {$\S_2$};
			\node at (-.1,.4+.5*2) {$\S_3$};
		\end{tikzpicture}
	\end{subfigure} 
	\caption{Routines \textsc{Define-Simplex-Subsets-A} (top left)
		and \textsc{Define-Simplex-Subsets-B} (bottom left),
		exemplary constructions for the 3-dimensional simplex~$H$
		with right-hand side~$ b = 4$ for the point~$h$
		with $ (h_1, h_2, h_3) = (1, 1.5, 0.8) $
		for \textsc{Define-Box-Subsets-A} (top right)
		and \textsc{Define-Box-Subsets-B} (bottom right).
		Via \textsc{Define-Box-Subsets-A}, we obtain
		$ h = 0.25 (4, 0, 0) + 0.375 (0, 4, 0) + 0.2 (0, 0, 4) + 0.175 (0, 0, 0) $
		while \textsc{Define-Box-Subsets-B} yields
		$ h = 0.3 (1, 2, 1) + 0.2 (1, 2, 0) + 0.5 (1, 1, 1) $.
		Those parts of the sets which belong to the same vertex
		are marked with the same colour;
		the numbers represent the height of each rectangle,
		\ie the coordinates of the vertices.
	}
	\label{fig:simplex}
	\null\vfill\null
\end{figure}
Via the first variant, the point $ h \in H $ is always written
as a convex combination of vertices of~$H$,
while in the second one the point may also be represented
using integral points inside the polytope.
By construction, both routines return sets~$ S_i $
with $ \muE{U}(S_i) = h_i $ for all $ i \in [n] $.
The inequalities defining~$H$ ensure that the combined width of the rectangles
fits into~$U$, and thus the requirements of \cref{thm:extended_zucker2}
are fulfilled in both cases.
This yields two different proofs for $ H = \conv(\FF) $
and shows the additional flexibility \cref{thm:extended_zucker2} offers.

\subsection{Convex-hull proofs for unbounded polyhedra}
\label{simplex:unb}

We continue with a modification of the previous simplex example,
where we demonstrate the difference it makes
to apply either \cref{thm:extended_zucker} or \cref{thm:extended_zucker2}
when showing integrality of an unbounded polyhedron.
Let $ \FF \coloneqq \set{x \in \Z^n \mid \sum_{i = 1}^n x_i \geq b,\, x \geq 0} $
and $ H \coloneqq \set{x \in \R^n \mid \sum_{i = 1}^n x_i \geq b,\, x \geq 0} $
with some $ b \in \N $.
The set characterizations for the two constraints defining~$ \FF $ and~$H$
can be stated as
\begin{alignat}{1}
	\sum_{i = 1}^n \phiE{U}(t, S_i) &\geq b \quad \forall t \in U,\label{char:simplex+cone1}\\
	\phiE{U}(t, S_i) &\geq 0 \quad \forall i \in [n], \forall t \in U.\label{char:simplex+cone2}
\end{alignat}
We can reuse routine \textsc{Define-Simplex-Subsets-B}
from \cref{fig:simplex} to construct adequate sets~$ S_i $
for \cref{thm:extended_zucker},
which proves the equivalence of~$ \conv(\FF) $ and~$H$.
An alternative representation of $ \FF$
is given by $ \FF = b\conv(e_1, \ldots, e_n) + \cone(e_1, \ldots, e_n) $.
Using the construction provided by routine \textsc{Define-Conv-Cone-Subsets}
in \cref{fig:simplex+cone},
we can invoke \cref{thm:extended_zucker2} and thus prove $ H = \conv(\FF) $
in an alternative fashion.
\begin{figure}[h]
	\centering
	\begin{subfigure}{\textwidth}
		\begin{algorithmic}[1]
			\Function{Define-Conv-Cone-Subsets}{}
				\State $ g \leftarrow h / \norm{h}_1 $
				\State $ (S_1, \ldots, S_n) \coloneqq
					\textsc{Define-Simplex-Subsets-B}
					\text{ applied to } g $
				\State $ v \leftarrow h - g $
				\For{$ i \in [n] $}
					\State $ S'_i \coloneqq ([\sum_{j = 1}^{i - 1} v_j,
						\sum_{j = 1}^i v_j), 1) $
				\EndFor
			\EndFunction
		\end{algorithmic}
	\end{subfigure}
	\par\vskip\floatsep
	\begin{subfigure}{.5\textwidth}
		\begin{tikzpicture}[xscale=4,yscale=1]
			\draw[thick] (0,0) -- (1,0);
			\draw[thick] (0,.1) -- (0,.-.1);
			\draw (.2,.05) -- (.2,-.05);
			\draw (.4,.05) -- (.4,-.05);
			\draw (.6,.05) -- (.6,-.05);
			\draw (.8,.05) -- (.8,-.05);
			\draw[thick] (1,.1) -- (1,.-.1);
			\node at (0,-.4) {$0$};
			\node at (1,-.4) {$1$};
			\draw[fill=blue!30] (0,0.2) rectangle ++(0.3,0.4) node[pos=.5] {$1$};
			
			\draw[fill=green!30] (0.3,0.2+.5*1) rectangle ++(0.45,0.4) node[pos=.5] {$1$};
			
			\draw[fill=red!30] (0.3+0.45,0.2+.5*2) rectangle ++(0.24,0.4) node[pos=.5] {$1$};
			
			\node at (-.1,.4+.5*0) {$\S_1$};
			\node at (-.1,.4+.5*1) {$\S_2$};
			\node at (-.1,.4+.5*2) {$\S_3$};
		\end{tikzpicture}
	\end{subfigure}
	\hfil
	\begin{subfigure}{.45\textwidth}
		\centering
		\begin{tikzpicture}[xscale=4,yscale=1]
			\draw[thick] (0,0) -- (1,0);
			\draw[thick] (0,.1) -- (0,.-.1);
			\draw (.2,.05) -- (.2,-.05);
			\draw (.4,.05) -- (.4,-.05);
			\draw (.6,.05) -- (.6,-.05);
			\draw (.8,.05) -- (.8,-.05);
			\draw[thick] (1,.1) -- (1,.-.1);
			\node at (0,-.4) {$0$};
			\node at (1,-.4) {$3$};
			
			\draw[fill=brown!30] (0,0.2) rectangle ++(0.3*0.7,0.4) node[pos=.5] {$1$};
			
			\draw[fill=purple!30] (0.3*0.7,0.2+.5*1) rectangle ++(0.3*1.05,0.4) node[pos=.5] {$1$};
			
			\draw[fill=cyan!30] (0.3*0.7+0.3*1.05,0.2+.5*2) rectangle ++(0.3*0.56,0.4) node[pos=.5] {$1$};
			
			\node at (-.1,.4+.5*0) {$\S'_1$};
			\node at (-.1,.4+.5*1) {$\S'_2$};
			\node at (-.1,.4+.5*2) {$\S'_3$};
		\end{tikzpicture}
	\end{subfigure}
	\caption{Routine \textsc{Define-Simplex+Cone-Subsets} (top),
		exemplary construction for the 3-dimensional polyhedron~$H$
		with right-hand side~$ b = 1$
		for the point~$h$ with $ (h_1, h_2, h_3) = (1, 1.5, 0.8) $
		for \textsc{Define-Conv-Cone-Subsets} (bottom).
		The latter decomposes~$h$ into $ h = g + v $,
		where $ g \approx 0.3 (1, 0, 0) + 0.45 (0, 1, 0) + 0.24 (0, 0, 1) $
		and $ v \approx 0.7 (1, 0, 0) + 1.05 (0, 1, 0) + 0.56 (0, 0, 1) $.
		The vector~$g$ is represented by a convex combination of the vertices
		of the polytopal part of~$H$, and $v$ as a conic combination of its rays.
		Those parts of the sets which belong to the same vertex or ray
		are marked with the same colour;
		the numbers represent their coordinates.}
	\label{fig:simplex+cone}
	\null\vfill\null
\end{figure}
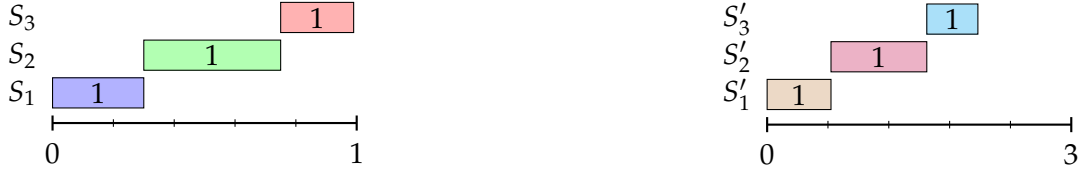
We can use both methods in order to prove the same statement.
However, we obtain different linear combinations
representing a given point~$h$.
\cref{thm:extended_zucker} gives us a convex combination
of arbitrary points in~$H$.
In contrast, \cref{thm:extended_zucker2} returns two sets of points,
vertices and rays, such that a convex combination of the vertices
plus a conic combination of the rays yields~$h$.
Depending on the problem at hand,
both strategies might be the one which is best suited
for a convex-hull proof.

\subsection{Convex-hull proofs for non-linear convex sets}

Finally, we show that our new criteria for convex-hull proofs
can also be used with non-polyhedral convex sets.
To do so, we use the example of the unit-ball in $ \R^2 $.
Let $ \FF \coloneqq \set{x \in \R^2 \mid (x_1 - 1)^2 + (x_2 - 1)^2 = 1} $
and $ H \coloneqq \set{x \in \R^2 \mid (x_1 - 1)^2 + (x_2 - 1)^2 \leq 1} $.
The set characterization for the quadratic constraint defining~$ \FF $ is given by
\begin{alignat}{1}
	(\phiE{U}(t, S_1) - 1)^2 + (\phiE{U}(t, S_2) - 1)^2 &= 1 \quad \forall t \in U. \label{char:circle}
\end{alignat}
Note that, unlike what Zuckerberg's original method allows, the set~$ \FF $
is not only infinite, as in the previous example, but even uncountable.

A set construction fulfilling the prerequisites of \cref{thm:extended_zucker}
is given by routine \textsc{Define-Ball-Subsets} in \cref{fig:circle}.
\begin{figure}[h]
	\centering
	\begin{subfigure}{.5\textwidth}
		\begin{algorithmic}[1]
			\Function{Define-Ball-Subsets}{}
				\State $ V \leftarrow \set{v \in \R \mid
					(v - 1)^2 + (h_2 - 1)^2 = 1} $
				\If{$ V = \set{v^1} $}
					\State $ S_1 \coloneqq ([0, 1), v^1]) $
					\State $ S_2 \coloneqq ([0, 1), h_2]) $
				\EndIf
				\If{$ V = \set{v^1, v^2} $, $ v^1 \neq v^2 $}
					\State $ \set{\lambda} \leftarrow
						\set{\lambda \mid \lambda v^1 + (1 -\lambda) v^2 = 1} $
					\State $ S_1 \coloneqq ([0, \lambda), v^1])
						\cup ([\lambda, 1), v^2]) $
					\State $ S_2 \coloneqq ([0, 1), h_2]) $
				\EndIf
			\EndFunction
		\end{algorithmic}
	\end{subfigure}
	\hfil
	\begin{subfigure}{.45\textwidth}
		\begin{tikzpicture}[xscale=4,yscale=1]
			\draw[thick] (0,0) -- (1,0);
			\draw[thick] (0,.1) -- (0,.-.1);
			\draw (.2,.05) -- (.2,-.05);
			\draw (.4,.05) -- (.4,-.05);
			\draw (.6,.05) -- (.6,-.05);
			\draw (.8,.05) -- (.8,-.05);
			\draw[thick] (1,.1) -- (1,.-.1);
			\node at (0,-.4) {$0$};
			\node at (1,-.4) {$1$};
			\draw[fill=orange!30] (0,0.2) rectangle ++(0.39,0.4) node[pos=.5] {$0.56$};
			\draw[fill=orange!30] (0,0.2+.5*1) rectangle ++(0.39,0.4) node[pos=.5] {$1.9$};
			
			\draw[fill=red!30] (0.39,0.2+.5*0) rectangle ++(0.61,0.4) node[pos=.5] {$1.44$};
			\draw[fill=red!30] (0.39,0.2+.5*1) rectangle ++(0.61,0.4) node[pos=.5] {$1.9$};
			
			\node at (-.1,.4+.5*0) {$\S_1$};
			\node at (-.1,.4+.5*1) {$\S_2$};
		\end{tikzpicture}
	\end{subfigure}
	\caption{Routine \textsc{Define-Ball-Subsets} (left),
		exemplary construction for the point~$h$
		with $ (h_1, h_2) = (1.1, 1.9) $ (right).
		The routine returns the representation
		$ h \approx 0.39 (0.56, 1.9) + 0.61 (1.44, 1.9) $.}
	\label{fig:circle}
	\null\vfill\null
\end{figure}
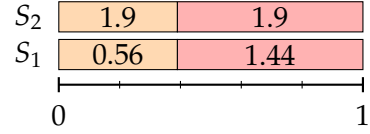
It computes two points on the boundary of the unit-ball
which have the same $y$-coordinate
and then calculates the corresponding coefficients
to represent~$h$ as a convex combination of the two.

\section{Set characterizations for integer problems}

In this section, we give two indicative convex-hull proofs
to illustrate the potential of our extended Zuckerberg framework.
We show that it can be applied to mixed-integer problems,
which the original method does not allow.
Furthermore, we show that it is well-suited to be used
with a non-fixed right-hand side, which leads to a new approach
to prove the total unimodularity of a matrix.

\subsection{Convex-hull proofs for mixed-integer problems}
To give a prominent example for the use of our extended Zuckerberg framework
in the mixed-integer case, we give a convex-hull proof
for the (single-item) uncapacitated lot-sizing problem (LS-U for short).
This problem asks for a cost-optimal production plan for a given product
over $n$~time periods to fulfil the customer demand $ d_j \in \R_+ $
in each period $ j \in [n] $
(see \cite{PochetWolsey2006} for an extensive introduction.)

The authors of \cite{KrarupBilde1977}
introduce the following extended formulation
for the feasible set of LS-U
(extended with respect to a straightforward formulation
with linearly-many variables, \cf \cite{PochetWolsey2006}):
let the variable $ w_{uj} \in \R_+ $ denote how much of the product
is produced in period $ u \in [n] $
for sale in the same or later period $ j \in [n] \setminus [u - 1] $.
Furthermore, variable $ y_u \in \F $
models the decision to perform any production in period $ u \in [n] $ or not.
Then we can represent the set of feasible production plans as
\begin{alignat}{3}
	\sum_{u = 1}^j w_{uj} &= d_j &
		\quad & \forall j \in [n],\label{equ:lot-sizing1}\\
	w_{uj} & \leq d_j y_u &
		\quad & \forall u \in [n], \forall j \in [n] \setminus [u - 1],
		\label{equ:lot-sizing2}\\
	w_{uj} & \in \R_+ &
		\quad & \forall u \in [n], \forall j \in [n] \setminus [u - 1],
		\label{equ:lot-sizing4}\\
	y_u & \in \F & \quad & \forall u \in [n].\label{equ:lot-sizing3}
\end{alignat}
Indeed, it is shown in \cite{KrarupBilde1977}
that the above model is integral,
\ie the set of solutions does not change
when relaxing $ y \in \F^n $ to $ y \in \FR^n $.
We will give an alternative proof based on \cref{thm:extended_zucker}.
\begin{theorem}[{\cite{KrarupBilde1977}}]
	Let $ P \coloneqq \conv \set{(y, w) \in \F^n \times \R^{n^2 - n}_+ \mid
		\cref{equ:lot-sizing1,equ:lot-sizing2}} $
	and $ H \coloneqq \set{(y, w) \in \FR^n \times \R^{n^2 - n}_+ \mid
		\cref{equ:lot-sizing1,equ:lot-sizing2}} $
	its linear relaxation.
	Then we have $ H = P $.
\end{theorem}
\begin{proof}
	The relation $ H \subseteq P $ can easily be seen.
	In order to prove the reverse,
	we transform the constraints defining $ P $
	into set characterizations:
	\begin{alignat}{3}
		\sum_{u = 1}^j \phiE{U}(S_{w_{uj}}, t) &= d_j
			& \quad & \forall j \in [n], \forall t \in U,
				\label{equ:lot-sizing1-char}\\
		\phiE{U}(S_{w_{uj}}, t) & \leq d_t\phiE{U}(S_{y_u}, t)
			& \quad & \forall\, 1 \leq u \leq j \leq n, \forall t \in U,
				\label{equ:lot-sizing2-char}\\
		\phiE{U}(S_{w_{uj}}, t) & \in \R_+
			& \quad & \forall\, 1 \leq u \leq j \leq n, \forall t \in U,
				\label{equ:lot-sizing4-char}\\
		\phiE{U}(S_{y_u}, t) & \in \F
			& \quad & \forall u \in [n], \forall t \in U.
				\label{equ:lot-sizing3-char}
	\end{alignat}
	For a given point $ h = (h_w, h_y) \in H $,
	a corresponding set construction is given
	in routine \textsc{Define-Lot-Sizing-Sets}
	in \cref{Fig:lotsizing}.
	\begin{figure}
		\begin{algorithmic}[1]
			\Function{Define-Lot-Sizing-Subsets}{}
				\For{$ u \in [n] $}
					\If{$ h_{y_u} \neq 0 $}
						\State $ S_{y_u} \coloneqq ([0, h_{y_u}), 1) $
					\Else
						\State $ S_{y_u} \coloneqq \emptyset $
					\EndIf
				\EndFor
				\For{$ u \in [n] $}
					\For{$ j \in [n] \setminus [u - 1] $}
						\If{$ h_{y_u} \neq 0 $}
							\State $ S_{w_{uj}} \coloneqq
								([0, h_{y_u}), h_{w_{uj}} / h_{y_u}) $
						\Else
							\State $ S_{w_{uj}} \coloneqq \emptyset $
						\EndIf
					\EndFor
				\EndFor
			\EndFunction
		\end{algorithmic}
		\centering
		\caption{Routine \textsc{Define-Lot-Sizing-Sets}}
		\label{Fig:lotsizing}
	\end{figure}
	In Lines 2--8, the routine places the sets for the $y$-variables
	such that \cref{equ:lot-sizing3-char} is satisfied.
	Then the $w$-variables are placed in Lines 9--17.
	The variables~$ w_{uj} $ get a non-empty set only if $ h_{y_u} \neq 0 $.
	The corresponding sets~$ S_{w_{uj}} $ are defined
	such that they have the same support as~$ S_{y_{u}} $.
	The construction satisfies \cref{equ:lot-sizing2-char,equ:lot-sizing1-char,equ:lot-sizing4-char}.
	Additionally, the defined sets fulfil $ \muE{U}(S_{w_{uj}}) = h_{w_{uj}} $
	for all $ 1 \leq u \leq j \leq n $
	and $ \muE{U}(S_{y_u}) = h_{y_u} $
	for all $ 1 \leq u \leq n $,
	which finishes the proof.
\end{proof}
Our Zuckerberg proof for the lot-sizing problem is an example
of the greedy proof strategy from \cref{sec:shortest-path},
now applied to the mixed-integer case.
In the \supplement, we give further such examples in the context of
mixed-integer models for piecewise linear functions.

\subsection{Showing total unimodularity via Zuckerberg's method}

Via our extension of Zuckerberg's method,
it is also possible to show the total unimodularity of a matrix
by using the following famous characterization of totally unimodular matrices.
\begin{theorem}[Hoffmann and Kruskal, \cite{HoffmanKruskal1956}]
	Let $ A \in \set{0, 1, -1}^{m \times n} $.
	Then $A$ is totally unimodular
	iff $ \set{x \in \R \mid Ax \leq b, x \geq 0} $
	has only integral vertices for all $ b \in \Z^n $.
	\label{hoffi:thm}
\end{theorem}
We will demonstrate the principle by reproving the well-known result
that the incidence matrix of a bipartite graph
is totally unimodular using Zuckerberg's method.

Let $ G = (V, E) $ be an undirected graph,
and let $ b \in \Z^{\card{E}} $ be an arbitrary integral vector.
Further, let~$P$ be the polytope defined as the convex hull
of all vectors $ x \in \N^{\card{E}} $ that satisfy
\begin{equation}
	x_i + x_j \leq b_{ij} \quad \forall \set{i, j} \in E.
	\label{inci:set-inq}
\end{equation}
The constraint matrix~$A$ corresponding to system~\cref{inci:set-inq}
is the transpose of a node-edge incidence matrix.
Its total unimodularity is stated in the following theorem,
for which we give a very simple proof based on \cref{thm:extended_zucker}.
\begin{theorem}
	Let $ P \coloneqq \conv \set{x \in \N^{\card{E}} \mid
		\cref{inci:set-inq}} $
	and $ H \coloneqq \set{x \in \R_+^{\card{E}} \mid
		\cref{inci:set-inq}} $
	its linear relaxation.
	Then we have $ P = H $.
	\label{inci:thm}
\end{theorem}
\begin{proof}
	The relation $ P \subseteq H $ is obvious.
	In order to prove $ H \subseteq P $,
	we transform constraint~\cref{inci:set-inq}
	into the set characterization
	\begin{alignat}{1}
		\phiE{U}(S_i, t) + \phiE{U}(S_j, t) \leq b_{ij}
			& \quad \forall \set{i, j} \in E, \forall t \in U.
			\label{inci:char1}
	\end{alignat}
	\WLOGc, we can assume $ b_{ij} \geq 0 $ for all $ \set{i, j} \in E $,
	since otherwise the polytope~$H$ is empty.
	For each point $ h \in H $, we then need to find sets~$ S_a $
	for all $ a \in A $ such that they fulfil $ \muE{U}(S_a) = h_a $
	and the above conditions hold.
	Let~$ W, Y \subseteq V $ be the two bipartite node sets of~$G$.
	The sets $ S_a $ are defined
	in routine \textsc{Define-Incidence-Matrix-Subsets}
	given in \cref{fig:zuckerberg-inci}.
	\begin{figure}
		\null\vfill\null
		\begin{algorithmic}[1]
			\Function{Define-Incidence-Matrix-Subsets}{}
				\For{each $ y \in Y $}
					\State $ S_y \coloneqq ([0, 1),
						\myfloor{h_y}) \cup ([0, h_y - \myfloor{h_y}), 1) $
				\EndFor
				\For{each $ w \in W $}
					\State $ S_w \coloneqq ([0, 1),
						\myfloor{h_w}) \cup ([1 - h_w + \myfloor{h_w}, 1), 1) $
				\EndFor
			\EndFunction
		\end{algorithmic}
		\caption{Routine \textsc{Define-Incidence-Matrix-Subsets}}
		\label{fig:zuckerberg-inci}
		\null\vfill\null
	\end{figure}
	From the above construction it is apparent that for each $ h \in H $
	the corresponding sets satisfy~\cref{inci:char1}.  
	Thus, we have proved $ H \subseteq P $.
\end{proof}
The desired result then follows from \cref{inci:thm,hoffi:thm}
by exploiting that total unimodularity is preserved under transposition.
\begin{corollary}
	Let $A$ be the node-edge incidence matrix of a bipartite graph.
	Then~$A$ is totally unimodular.
\end{corollary}
We think that the possibility to consider arbitrary right-hand sides
in an algorithmic fashion makes the Zuckerberg approach a valuable tool
for proving total unimodularity.
A further example for this concept is given in the \supplement.

\section{Extensions of Zuckerberg's method for graphs of functions}

In \cite{gupte2020extended}, Zuckerberg's method was adapted
to characterize the convex hull of the graphs
of certain bilinear functions defined over the unit cube.
Using our extended framework for convex-hull proofs
from \cref{sec:zuckerberg_extended},
we will generalize these results in a twofold manner.
Firstly, we extended the machinery introduced there
for bilinear functions to general boolean functions.
This allows us to treat common functions like the $\max$-function.
In addition, we generalize the applicability
of Zuckerberg's method to non-box domains,
such that it works with functions defined over any $0$/$1$-polytope.
Secondly, we will derive a criterion to prove convex-hull results
for the convex hull of graphs of bilinear functions
over general polytopal domains.

\subsection{Extension for boolean functions over $0$/$1$-polytopes}
\label{Sec:Boolean_Extension}

Let $ \FF \subseteq \R^n $ be a finite set of points,
and let $ T \coloneqq \conv(\FF) $ be their convex hull.
We will now consider functions $ f\colon \FF \rightarrow \R $ of the form
\begin{equation*}
	f(x) = \sum_{i = 1}^k a_i \Psi_i(x_1, \ldots x_n),
\end{equation*} 
with $ \func{\Psi_i}{\FF}{\R} $ and $ a_i \in \R $ for $ i \in [k] $.
The convex hull of the graph of~$f$ is the set
\begin{equation*}
	X(f) \coloneqq \conv \SSet{(x, z) \in T \times \R }{z = f(x)}.
\end{equation*}
Further, let the two functions $ \vex[f]\colon T \rightarrow \R $
and $ \cav[f]\colon T \rightarrow \R $,
denoting the convex and the concave envelope of~$f$ over $T$, respectively,
be defined as
\begin{align*}
	\vex[f](x) &\coloneqq \min\SSet{z \in \R}{(x, z) \in X(f)},\\
		\cav[f](x) &\coloneqq \max\SSet{z \in \R}{(x, z) \in X(f)},
\end{align*}
so that we have
\begin{equation*}
	X(f) = \SSet{(x, z) \in T \times \R}{\vex[f](x) \leq z \leq \cav[f](x)}.
\end{equation*}
Introducing variables $ y_i $
to represent the products $ \Psi_i(x_1, \ldots x_n) $,
we are interested in describing~$ X(f) $
in terms of the $x$- and $y$-variables.
To be more precise, we define a function
$ \pi[f]\colon \R^n \times \R^k \to \R^{n + 1} $ via
\begin{equation*}
	\pi[f](x, y) = \left(x, \sum_{i = 1}^k a_i y_i\right)
\end{equation*}
and extend it to the power set of $ \R^n \times \R^k $ in a canonical fashion:
\begin{equation*}
	\pi[f](P) = \SSet{\pi[f](x, y)}{(x, y) \in P}
\end{equation*}
for every $ P \subseteq \R^n \times \R^k $.
For a polytope~$P$, let the functions $ \LB_{P}[f]\colon T \to \R $
and $ \UB_{P}[f]\colon T \to \R $ be defined as
\begin{alignat*}{1}
	\LB_{P}[f](x) &= \min\SSet{\sum_{i = 1}^k a_i y_i}{(x, y) \in P}
		\hspace{-0.05cm} = \hspace{-0.05cm} \min\SSet{z \in \R}{(x, z) \in \pi[f](P)},\\
	\UB_{P}[f](x) &= \max\SSet{\sum_{i = 1}^k a_i y_i}{(x, y) \in P}
		\hspace{-0.05cm} = \hspace{-0.05cm} \max\SSet{z \in \R}{(x, z) \in \pi[f](P)},
\end{alignat*}
respectively, so that 
\begin{equation*}
	\pi[f](P) = \SSet{(x, z) \in T \times \R}
		{\LB_{P}[f](x) \leq z \leq	\UB_{P}[f](x)}.
\end{equation*}
The goal is to give a criterion which allows to prove
$ X(f) = \pi[f](P) $ for some given function~$f$ and polytope~$P$,
which is equivalent to $ \vex[f](x) = \LB_{P}[f](x) $
and $ \cav[f](x) = \UB_{P}[f](x) $ for all $ x \in T $.
To this end, we define the set
\begin{equation*}
	Z(x) \coloneqq \SSet{(S_1, \ldots, S_n) \in \LL^n }{\begin{split}
		\mu(S_i) = x_i \quad \forall i \in [n],\\
		\varphi(t, S_1, \ldots, S_n) \in \FF\quad \forall t \in U
		\end{split}}.
\end{equation*}
It contains all tuples of admissible sets $ S_1, \ldots, S_n $
which express some point $ x \in \FR^n $
via the vertices of $ \FF $ using Zuckerberg's certificate.
Finally, let the function $ \Omega\colon \LL^n \times \R^{\FF} \to U $,
\begin{equation*}
	\Omega(S_1, \ldots, S_n, \Psi) \coloneqq
		\mu(\SSet{t \in U}{\Psi(\varphi(t, S_1, \ldots, S_n)) = 1})
\end{equation*}
measure the size of the support of $ \Psi \circ \varphi $
for some $ \Phi\colon \FF \to \R $ and some fixed $ (S_1, \ldots S_n) \in \LL^n $.
The proof of $ \pi[f](P) = X(f) $ can be split up
into $ \pi[f](P) \subseteq X(f) $ and $ X(f) \subseteq \pi[f](P) $.
The first inclusion is often comparably easy to prove,
and for the validity of the second inclusion
we give the following criterion. 
\begin{theorem}
	\label{cor:set_interpretation3}
	If $ \FF \subseteq \F^n $ and $ f = \sum_{i = 1}^k a_i \Psi_i $,
	with $ \func{\Psi_i}{\F^n}{\F} $ and $ a_i \in \R $ for $ i \in [k] $,
	we have
	\begin{align*}
		\vex[f](x) &= \min\SSet{\sum_{i \in [k]}
			a_i \Omega(S_1, \ldots, S_n, \Psi_i)}
			{(S_1, \ldots, S_n) \in Z(x)},\\
		\cav[f](x) &= \max\SSet{\sum_{i \in [k]}
			a_i \Omega(S_1, \ldots, S_n, \Psi_i)}
			{(S_1, \ldots, S_n) \in Z(x)}
	\end{align*}
	for all $ x \in T $.
	In particular, for a polytope $ P \subseteq \R^{n + k} $
	with $ \pi[f](P) \subseteq X(f) $
	we have $ \pi[f](P) = X(f) $ iff for every $ x \in T $
	there are sets $ (S_1, \ldots, S_n) \in Z(x) $
	and $ (S'_1, \ldots, S'_n) \in Z(x) $ with
	\begin{align*}
		\sum_{i \in [k]} a_i \Omega(S_1, \ldots, S_n, \Psi_i) &= \LB_P[f](x),\\
		\sum_{i \in [k]} a_i \Omega(S'_1, \ldots, S'_n, \Psi_i) &= \UB_P[f](x).
	\end{align*}
\end{theorem}
\cref{cor:set_interpretation3} gives us a Zuckerberg-type characterization
of $ \vex[f] $ and $ \cav[f] $.
To apply it, we need to design for a general point $ x \in T $
the sets $ S_1, \ldots, S_n \in Z(x) $
such that we minimize $ \sum_{i \in [k]} a_i \Omega(S_1, \ldots, S_n, \Psi_i) $
and $ S'_1, \ldots, S'_n \in Z(x) $
such that we maximize $ \sum_{i \in [k]} a_i \Omega(S'_1, \ldots, S'_n, \Psi_i) $. 
The proof of \cref{cor:set_interpretation3}
is given
in the \supplement.

The expression $ \Omega(S_1, \ldots, S_n, \Psi) $
can be made more tractable when some specific functions $ \Psi $ is given.
Consider, for instance, $ \Psi(x_1, x_2) = x_1 x_2 $.
Then we can simplify:
\begin{equation*}
	\Omega(S_1, S_2, \Psi) =
		\mu(\SSet{t \in U}{\phi(t, S_1)\phi(t, S_2) = 1}) = \mu(S_1 \cap S_2).
\end{equation*}
We exemplarily give similar representations for $ \Omega(S_1, \ldots, S_n, \Psi) $
for some further Boolean functions in \cref{tb:set-char4}.
A specialization of \cref{cor:set_interpretation3}
for the case $ f(x) = \sum_{1 \leq i < j \leq n} a_{ij} x_i x_j $ and $ \FF = \F^n $
was proved in \cite{gupte2020extended}
using the above simplification.
In the following, we will demonstrate how to use \cref{cor:set_interpretation3}
to give convex-hull proofs for more general domains and functions.

\begin{table}[h]
	\centering
	\caption{Simplifications of $ \Omega(S_1, \ldots, S_n, \Psi) $
		for specific boolean $ \Psi $ functions}
	\label{tb:set-char4}
	\begin{tabular}{rrr}
		\toprule
		$ \Psi $ & Corresp.\ Boolean operator & Simplified $ \Omega $\\
		\midrule
		$ \min(x_i, x_j) $ & AND & $ \mu(S_i \cap S_j) $\\
		$ \max(x_i, x_j) $ & OR & $ \mu(S_i \cup S_j) $\\
		$ x_i \text{ XOR } x_j $ & XOR & $ \mu((S_i \cap \bar{S}_j)
		\cup (\bar{S}_i \cap S_j)) $\\
		\midrule
		$ \min(x_1, \ldots, x_n) $ & AND & $ \mu(S_1 \cap \ldots \cap S_n) $\\
		$ \max(x_1, \ldots, x_n) $ & OR & $ \mu(S_1 \cup \ldots \cup S_n) $\\
		\bottomrule
	\end{tabular}
\end{table}

\subsubsection{Convex-hull proofs for polytopal domain}

Generalizing an example given in \cite{gupte2020extended} for unit-box domains,
we show here how to characterize the McCormick-relaxation
of the product of two binary variables over a non-box binary polytope.
Let $ \FF \coloneqq \set{x \in \F^2 \mid x_1 + x_2 \geq 1} $
and $ \func{f}{\FF}{\F},\, f(x_1, x_2) = x_1 x_2 $, and let 
\begin{equation*}
	P \coloneqq \SSet{(x_1, x_2, z) \in \FR^3}
		{z \leq x_1,\, z \leq x_2,\, z \geq x_1 + x_2 - 1,\, x_1 + x_2 \geq 1}.
\end{equation*}
The direction $ \pi[f](P) \subseteq X(f) $ can easily be verified
by checking if the extreme points of~$ X(f) $,
namely $ (0, 1, 0), (1, 0, 0) $ and $ (1, 1, 1) $, are feasible for~$P$.
For the reverse direction, we plug in the simplification of $ \Omega $
for $f$ given in \cref{tb:set-char4} into \cref{cor:set_interpretation3}.
We deduce that we need to find two sets~$ S_1 $ and~$ S_2 $ which fulfil
\begin{equation*}
	\mu(S_1 \cap S_2) \leq \min\set{x_1,x_2}
	\label{mccormickpolytope1}.
\end{equation*}
It follows $ \cav[f](x) \leq \min\set{x_1, x_2} $,
and with $ S_1 \coloneqq [0, x_1) $ and $ S_i \coloneqq [0, x_2) $
we see that this bound is attained for all $ x \in \conv(\FF) $.
Therefore, the concave envelope of~$f$
is given by the inequalities $ z \leq x_1 $ and $ z \leq x_2 $.
Similarly,
\begin{equation*}
	\mu(S_1 \cap S_2) \geq x_1 + x_2 - 1
	\label{mccormickpolytope2} 
\end{equation*}
leads to the bound $ \vex[f](x) \geq x_1 + x_2 - 1 $,
and with $ S_1 \coloneqq [0, x_1) $ and $ S_2 \coloneqq [1 - x_2, 1) $
it is attained for all $ x \in \conv (\FF) $.
Thus, the convex envelope of~$f$ is given by $ z \geq  x_1 + x_2 - 1 $.
Finally, the constraint $ x_1 + x_2 \geq 1 $
is needed for the initial restriction of the domain.

\subsubsection{Convex-hull proofs for general functions}

Now we present an example for the $ \max $-function,
which shows that our framework is applicable to more general functions
than the bilinear functions studied in \cite{gupte2020extended}.
We consider the case $ \func{f}{\F^n}{\F},\, f(x) = \max(x_1, \ldots, x_n) $
and
\begin{equation*}
	P \coloneqq \SSet{(x_1, \ldots, x_n, z) \in \FR^{n + 1}}
		{z \leq x_1 + \ldots + x_n,\, z \leq 1,\, z \geq x_i\, \forall i \in [n]}.
\end{equation*}
As in the previous example, it is straightforward
to verify $ \pi[f](P) \subseteq X(f) $.
For the converse, we use the simplification for~$f$
given in \cref{tb:set-char4} to see that the sets $ S_1, \ldots, S_n $
need to fulfil
\begin{equation*}
	\mu(S_1 \cup \ldots \cup S_n) \leq \min\lrset{x_1 + \ldots + x_n, 1}
	\label{max1}
\end{equation*}
in order to satisfy \cref{cor:set_interpretation3}.
Consequently, we have $ \cav[f](x) \leq \min\set{x_1 + \ldots + x_n, 1} $.
With $ S_1 \coloneqq [0, x_1) $
and $ S_i \coloneqq [x_{i - 1}, x_{i - 1} + x_i) \mod 1 $
for $ i \in \set{2, \ldots, n} $,
this bound is attained for all $ x \in \FR^n $.
Therefore, the concave envelope of~$f$
is given by $ z \leq x_1 + \ldots + x_n $ and $ z \leq 1 $.
Furthermore, from
\begin{equation*}
	\mu(S_1 \cup \ldots \cup S_n) \geq \max\set{x_1, \ldots, x_n}
\label{max2}
\end{equation*}
we obtain the bound $ \vex[f](x) \geq \max\set{x_1, \ldots, x_n} $,
and setting $ S_i \coloneqq [0, x_i) $ for $ i \in \set{1, \ldots, n} $
makes it tight for all $ x \in \FR^n $.
Thus, the convex envelope of~$f$ is given by $ z \geq x_i $ for $ i \in \set{1, \ldots, n} $.

\subsection{Extension for bilinear functions over general polytopes}

Finally, we derive a generalization of the results from \cite{gupte2020extended}
which allows us to compute the convex hull of the graph of a bilinear function
over a general polytopal domain~$T$.
We start by defining the set
\begin{equation*}
	\ZE(x) \coloneqq \SSet{(S_1, \ldots, S_n) \in (\LLE{U})^n}
		{\begin{split}
			\muE{U}(S_i) = x_i \text{ for all } i \in [n],\\
			\varphiE{U}(t, S_1, \ldots, S_n) \in \FF\, \forall t \in U
			\end{split}},
\end{equation*}
which contains all admissible sets $ S_1, \ldots, S_n $
for \cref{cor:set_interpretation3}.
Furthermore, let the two functions $ \bar{\psi}_-, \bar{\psi}_+ \colon P \to \R $ with
\begin{align*}
	\bar{\psi}_-(x) &= \min\SSet{\sum_{\xi \in \FF}
		\muE{U}(\LE{U}_\xi(S_1, \ldots, S_n)) f(\xi)}
		{(S_1, \ldots, S_n) \in \ZE(x)},\\
	\bar{\psi}_+(x) &= \max\SSet{\sum_{\xi \in \FF}
		\muE{U}(\LE{U}_\xi(S_1, \ldots, S_n)) f(\xi)}
		{(S_1, \ldots, S_n) \in \ZE(x)}
\end{align*}
encode the convex and concave envelope, respectively,
in a Zuckerberg fashion.
We now derive an auxiliary representation of $ X(f) $ in terms of these two functions.
\begin{lemma}
	\label{thm:zucker_function2}
	For every function $ f\colon \FF \to \R $, we have
	\begin{equation*}
		X(f) = \SSet{(x, z) \in T \times \R}
			{\bar{\psi}_-(x) \leq z \leq \bar{\psi}_+(x)}.
	\end{equation*}  
\end{lemma}
\begin{proof}
	First, assume $ (x, z) \in X(f) $.
	This means
	\begin{equation*}
		(x, z) = \sum_{k = 1}^{\card{\FF}} \lambda_k (\xi^k, f(\xi^k))
	\end{equation*}
	for some $ \lambda_k \geq 0 $ for $ k = 1, \ldots, \card{\FF} $
	with $ \sum_{k = 1}^{\card{\FF}} \lambda_k = 1 $
	and a fixed ordering $ \xi^1, \ldots, \xi^{\card{\FF}} $ of $ \FF $.
	The sets $ S_i \in \LLE{U} $ with $ \muE{U}(S_i) = x_i $
	are defined exactly as in the proof of Theorem~\ref{thm:extended_zucker}:
	for the partition $ U = I_1 \cup \ldots \cup I_{\card{\FF}} $
	with $ I_1 = [0, \lambda_1) $
	and $ I_k = [\lambda_1 + \ldots + \lambda_{k - 1},
		\lambda_1 + \ldots + \lambda_k) $
	for $ k \in \set{2, \ldots, \card{\FF}} $, we set
	\begin{equation*}
		S_i = \bigcup_{k:\, \xi^k_i \neq 0}(I_k, \xi^k_i).
	\end{equation*}
	For every $ k = 1, \ldots, \card{\FF} $,
	we have $ \LLE{U}_{\xi^k}(S_1,\ldots, S_n) = I_k $,
	and consequently $ \muE{U}(\LLE{U}_{\xi^k}(S_1, \ldots, S_n)) = \lambda_k $.
	With
	\begin{equation*}
		z = \sum_{k = 1}^{\card{\FF}} \lambda_k \psi(\xi^k)
			= \sum_{k = 1}^{\card{\FF}} \muE{U}
				(\LLE{U}_{\xi^k}(S_1, \ldots, S_n)) \psi(\xi^k),
	\end{equation*}
	it follows that $ \bar{\psi}_-(x) \leq z \leq \bar{\psi}_+(x) $.
	
	For the converse, assume $ \bar{\psi}_-(x) \leq z \leq \bar{\psi}_+(x) $
	and let $ (S_1, \ldots, S_n) $ and $ (S'_1, \ldots, S'_n) $
	be optimizers for the minima and maxima defining $ \bar{\psi}_-(x) $
	and $ \bar{\psi}_+(x) $ respectively.
	We write $ z = t\bar{\psi}_-(x) + (1 - t) \bar{\psi}_+(x) $
	for some $ t \in [0, 1] $ and set
	\begin{equation*}
		\lambda(\xi) \coloneqq t\muE{U}(\LLE{U}_{\xi}(S_1, \ldots, S_n))
			+ (1 - t) \muE{U}(\LLE{U}_{\xi}(S'_1, \ldots, S'_n))
	\end{equation*}
	for all $ \xi \in \FF $.
	This yields a representation of $ (x, z) $ as the convex combination
	\begin{align*}
		(x, z) = \sum_{\xi\in\FF} \lambda(\xi)(\xi, \psi(\xi)).
		\\[-2\baselineskip]
	\end{align*}
\end{proof}
With \cref{thm:zucker_function2},
we can express the convex hull of the graph of a function
via functions defined over Zuckerberg sets.
We will now make these abstract expressions
more concrete for the case of bilinear functions.
For this purpose, we consider an arbitrary bilinear function
\begin{equation}
	\func{f}{\FF}{\R}, \quad f(x) = \sum_{ij \in E}^k a_{ij} x_i x_j
	\label{Eq:Bilinear_Function}
\end{equation}
with coefficients $ a_{ij} \in \R $
and a subset $ E \subseteq \set{(i, j) \in \N^2 \mid 1 \leq i < j \leq n } $.
Furthermore, we define the function
$ \func{M}{(\LLE{U})^2}{\R} $ in \cref{Fig:M},
which measures a kind of generalized overlap between two sets $ S_1, S_2 \in \LLE{U} $.
\begin{figure}[h]
	\begin{algorithmic}[1]
		\Function{$M$}{$ S_1, S_2 $}
		\State $ r \leftarrow 0 $
		\State Let $ (R^1_1, \ldots R^1_l) $ be any ordering
			of the rectangles defining $ S_1 $
		\State Let $ (R^2_1, \ldots, R^2_p) $ be any ordering
			of the rectangles defining $ S_2 $		
		\For{$ R^1 \in (R^1_1, \ldots, R^1_l) $}
			\For{$ R^2 \in (R^2_1, \ldots, R^2_p) $}
				\State $ g \leftarrow \SSet{(t, 1) \in U \times \set{1}}
					{y(t, R^1) = 1, y(t, R^2) = 1} $
				\State $ r \leftarrow r + z(R^1) z(R^2) \muE{U}(g) $
			\EndFor
		\EndFor
		\State \Return{$r$}
		\EndFunction
	\end{algorithmic}
	\centering
	\caption{The function $ \func{M}{(\LLE{U})^2}{\R} $}
	\label{Fig:M}
\end{figure}
For the particular case of a bilinear function~$f$,
\cref{thm:zucker_function2} yields the following characterizations
of $ \cav[f](x) $ and $ \vex[f](x) $.
\begin{theorem}
	\label{cor:set_interpretation2}
	For the bilinear function~$f$ of the form~\cref{Eq:Bilinear_Function},
	we have
	\begin{align*}
		\vex[f](x) &= \min\SSet{\sum_{ij \in E} a_{ij} M(S_i, S_j)}
			{(S_1, \ldots, S_n) \in \ZE(x)},\\
		\cav[f](x) &= \max\SSet{\sum_{ij \in E} a_{ij} M(S_i, S_j)}
			{(S_1, \ldots, S_n) \in \ZE(x)}
	\end{align*}
	for all $ x \in T $.
	In particular, for a polytope $ P\subseteq \R^{n(n + 1) / 2} $
	with $ \pi[f](P) \subseteq X(f) $, we have $ \pi[f](P) = X(f) $
	if and only if for every $ x \in T $ there are sets
	$ (S_1, \ldots, S_n) \in \ZE(x) $ and $ (S'_1, \ldots, S'_n) \in \ZE(x) $ with
	\begin{align*}
		\sum_{ij \in E} a_{ij} M(S_i, S_j) &= \LB_P[f](x),\\
		\sum_{ij \in E} a_{ij} M(S'_i, S'_j) &= \UB_P[f](x).
	\end{align*}
\end{theorem}
\begin{proof}
	We observe that
	\begin{align*}
		\sum_{\xi \in \FF} & \muE{U} (\LLE{U}_{\xi}(S_1, \ldots, S_n)) f(\xi)
			= \sum_{\xi \in \FF} \muE{U} (\LLE{U}_{\xi}(S_1, \ldots, S_n))
				\sum_{ij \in E} a_{ij} \xi_i \xi_j\\
			&= \sum_{ij \in E} a_{ij} \sum_{\xi \in \FF} \muE{U}
				(\LLE{U}_\xi(S_1, \ldots, S_n)) \xi_i \xi_j\\
			&= \sum_{ij \in E} a_{ij} M(S_i, S_j).\\[-2\baselineskip]
	\end{align*}
\end{proof}
\cref{cor:set_interpretation2} allows us to give compact representations
for $ \vex[f] $ and $ \cav[f] $ over polytopal domains.
To do so, we need to design suitable sets $ S_1, \ldots, S_n \in \ZE $
such that $ \sum_{ij \in E} a_{ij} M(S_i, S_j) $ is minimized
and $ S'_1, \ldots, S'_n \in \ZE $
such that $ \sum_{ij \in E} a_{ij} M(S'_i, S'_j) $ is maximized,
both for an arbitrary point $ x \in \R^n $.

\subsubsection{Convex-hull proofs for non-$0$/$1$ domain}

We can use \cref{cor:set_interpretation2} to prove again
that $ X(f) $ for $ f(x_1, x_2) = x_1 x_2 $
is given by the McCormick-inequalities.
However, we will now do this over the bounds $ 0 \leq x_1 \leq u_1 $
and $ 0 \leq x_2 \leq u_2 $ for $ u_1, u_2 \geq 0 $
instead of the unit bounds as in \cref{sec:mc-cormick1}.
Let
\begin{equation*}
	P \coloneqq \SSet{(x_1, x_2, z) \in \R^3}
		{z \leq u_2 x_1, z \leq u_1 x_2, z \geq 0,
			z \geq u_2 x_1 + u_1 x_2 - u_1 u_2}.
\end{equation*}
The direction $ \pi[f](P) \subseteq X(f) $ can easily be verified.

For the reverse direction, we conclude from
\begin{equation}
	M(S_1, S_2) \leq u_1 u_2 \min\lrset{\frac{1}{u_1} x_1, \frac{1}{u_2} x_2}
	\label{mc-cormick2}
\end{equation}
that $ \cav[f](x) \leq \set{(1 / u_1) x_1, (1 / u_2) x_2} $,
and with $ S_1 = ([0, x_1/u_1), u_1) $, $ S_2=([0, x_2/u_2), u_2) $
this bound is attained for all $ x_1 \in [0, u_1] $, $ x_2 \in [0, u_2] $.
Therefore, the concave envelope is given by $ z \leq u_2 x_1 $ and $ z \leq u_1 x_2 $.
In a similar fashion, it follows from
\begin{equation}
	M(S_1, S_2) \geq u_1 u_2 \max\lrset{0, \frac{1}{u_1} x_1 + \frac{1}{u_2} x_2 - 1}
	\label{mc-cormick3}
\end{equation}
that $ \vex[f](x) \geq \max\set{0, (1 / u_1) x_1 + (1 / u_2) x_2 - 1} $.
By choosing the sets $ S_1 = ([0, x_1 / u_1), u_1) $, $ S_2 = ([1 - x_2 / u_2, 1), u_2) $
we can show that this bound is attained for all $ x_1 \in [0, u_1] $, $ x_2 \in [0, u_2] $.
Thus, the convex envelope is given by $ z \geq 0 $ and $ z \geq u_2 x_1 + u_1 x_2 - u_1 u_2 $.

\section{Conclusion}
We have presented a vastly simplified framework
for Zuckerberg's geometric proof technique for convex-hull results.
By restating the method in terms of our notion of set characterizations,
we were able to accomplish several benefits.
Firstly, we have identified three major strategies
one can pursue in Zuckerberg-type convex-hull proofs.
This underlines the high flexibility in devising algorithmic schemes the method offers.
Secondly, we have significantly extended the expressive power of Zuckerberg's technique
by basing it on a different underlying subset algebra.
It can now be used to characterize the convex hulls of general convex sets,
including, but not limited to integer polyhedra.
Using this extension, we give characterizations of the convex hull
of Boolean and bilinear functions over polytopal domains.
Finally, we have given a variety of indicative examples for the use of our framework
with the intention to convey the ideas as hands-on as possible.

We find it a very interesting avenue for future research
to develop further algorithmic strateg-ies for Zuckerberg proofs
and to extend the scope of those we have introduced.
For example, one could not only consider linear programs
but pass to (mixed-)integer ones when following the technique using feasibility subproblems.
This might entail the consideration of minimally infeasible subsystems
to verify the set characterizations.

Altogether, we make a strong case for the canonization of Zuckerberg's proof technique
in standard text books on integer programming and polyhedral combinatorics.
We are certain it will enable many more interesting convex-hull results in the future.

\section*{Acknowledgements}
\label{sec:acknowledgements}

We thank Alexander Martin and
Thomas Kalinowski for our fruitful discussions on the topic.
Futhermore, we acknowledge financial support
by the Bavarian Minis\-try of Economic Affairs, Regional Development and Energy
through the Center for Analytics -- Data -- Applications (ADA-Center)
within the framework of \qm{BAYERN DIGITAL~II}.

\bibliographystyle{alpha}
\bibliography{Literature}

\section{Online supplement}
In this online supplement, we give several further convex-hull proofs
as examples for our concept of set characterizations
and our convex extension of Zuckerberg's proof scheme in practice.

We start by giving further examples of convex-hull proofs
using the greedy proof strategy from \cref{sec:shortest-path}.
In \cref{Sec:Staircase}, we consider CPMC under staircase compatibility,
which is a second polynomial-time solvable subcase of CPMC.
Then we treat the stable-set problem on bipartite graphs
in \cref{Sec:Bipartite_Stable_Set}.
This is followed by two further examples for the use
of our extended Zuckerberg proof scheme in mixed-integer problems.
These are mixed-integer models for piecewise linear functions in \cref{Sec:Piecewise}
and the total unimodularity of interval matrices in \cref{Sec:TU_Interval}.
Finally, in \cref{sec:proofs-graph},
we give the proof for \cref{cor:set_interpretation3} from \cref{Sec:Boolean_Extension}.

\subsection{CPMC under staircase compatibility}
\label{Sec:Staircase}

Staircase compatibility is a special case of CPMC
which arises especially in scheduling applications
with precedence constraints (see \eg \cite{schwindt2015handbook,benders2020}),
but also when considering flow problems with piecewise linear costs, for example
(\cite{liers2016structural}).
Here, each subset of nodes $ V_i \in \VV $ is equipped with a total order $ <_i $,
which we assume to be the case in the following.
\begin{definition}(Subgraphs~$ \Gij $, staircase ordering, staircase partition)
	For any two subsets~$ V_i, V_j \in \VV $ with $ i \neq j $,
	we write
	\begin{equation*}
		\Gij \coloneqq (V_i \cup V_j, \Eij)
	\end{equation*}
	for the subgraph of~$G$ induced by $ V_i \cup V_j $,
	where $ \Eij $ is the corresponding edge set.
	Note that all subgraphs~$ \Gij $ are bipartite.
		
	For simplicity, we refer to the set $ \set{<_1, \ldots, <_m} $
	of total orders as an \emph{ordering} on~$G$ if~$ \VV $ is clear from the context.
	An ordering on~$G$ is \emph{staircase} if for all subgraphs~$ \Gij $ of~$G$
	the two conditions
	\begin{equation}
		\begin{split}
			u \in V_i \wedge v_1 <_j v_2 <_j v_3 \in V_j
				&\wedge \set{u, v_1}, \set{u, v_3} \in \Eij\\
				&\Rightarrow \set{u, v_2} \in \Eij,
		\end{split}
		\tag{SC1}
		\label{Eq:SC1}
	\end{equation}
	\begin{equation}
		\begin{split}
			u_1 <_i u_2 \in V_i \wedge v_1 <_j v_2 \in V_j
				&\wedge \set{u_1, v_2}, \set{u_2, v_1} \in \Eij\\
				&\Rightarrow \set{u_1, v_1}, \set{u_2, v_2} \in \Eij
		\end{split}
		\tag{SC2}
		\label{Eq:SC2}
	\end{equation}
	hold.
	We call $ \VV $ \emph{staircase} if there exists a staircase ordering on $ \VV $.
	\label{def:staircase}
\end{definition}
Condition~\cref{Eq:SC1} ensures that the neighbourhoods of all vertices are continuous
with respect to the total order,
whereas~\cref{Eq:SC2} yields a kind of monotonicity on the edge set $ E_{ij} $.
The two conditions are illustrated in Figure~\ref{fig:def_staircase}.
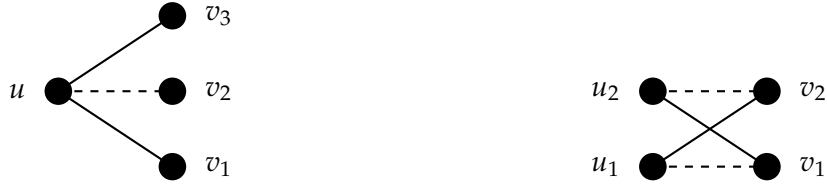
\begin{figure}[h]
	\begin{subfigure}[b]{0.48\textwidth}
		\centering
		\begin{tikzpicture}
			\draw[thick] (0,1) -- (1.5,2);
			\draw[thick, dashed] (0,1) -- (1.5,1);
			\draw[thick] (0,1) -- (1.5,0);
	
			\draw node[fill, circle, label={[label distance=3]180:{$u$}}] at (0,1) {};
			\draw node[fill, circle, label={[label distance=3]0:{$v_1$}}] at (1.5,0) {};
			\draw node[fill, circle, label={[label distance=3]0:{$v_2$}}] at (1.5,1) {};
			\draw node[fill, circle, label={[label distance=3]0:{$v_3$}}] at (1.5,2) {};
		\end{tikzpicture}
	\end{subfigure}
	\begin{subfigure}[b]{0.48\textwidth}
		\centering
		\begin{tikzpicture}
			\draw[thick] (0,0) -- (1.5,1);
			\draw[thick] (0,1) -- (1.5,0);
			\draw[thick, dashed] (0,0) -- (1.5,0);
			\draw[thick, dashed] (0,1) -- (1.5,1);
	
			\draw node[fill, circle, label={[label distance=3]180:{$u_1$}}] at (0,0) {};
			\draw node[fill, circle, label={[label distance=3]180:{$u_2$}}] at (0,1) {};
			\draw node[fill, circle, label={[label distance=3]0:{$v_1$}}] at (1.5,0) {};
			\draw node[fill, circle, label={[label distance=3]0:{$v_2$}}] at (1.5,1) {};
		\end{tikzpicture}
	\end{subfigure}
	\caption{Illustration of the two staircase conditions
		\cref{Eq:SC1} (left) and \cref{Eq:SC2} (right).
			If the solid edges are contained in the graph,
			the dashed ones must be contained as well.}
	\label{fig:def_staircase}
\end{figure}
The name staircase compatibility is motivated by the fact
that the node adjacency matrices corresponding to the bipartite graphs~$ \Gij $
feature a completely dense staircase form
if the rows and columns are arranged according to the orders~$ <_i $ and~$ <_j $.
The term \emph{completely dense} refers to the ones in each row
forming a consecutive block not interrupted by any zeroes.
It is easy to see that~\eqref{Eq:SC1}
is implied by~\eqref{Eq:SC2} if $ \Gij $ does not contain vertices with degree zero.
The problem of recognizing the partition~$ \VV $ to be staircase
is addressed in \cite{staircase2020}.

As the authors of \cite{staircase2018} have shown,
the CPMC polytope can be represented via a totally unimodular system
of polynomial size if the considered instance of CPMC has the staircase property.
In the following, we will give an alternative, shorter proof
for their convex-hull result by using Zuckerberg's method.
To facilitate notation, let
$ \min(v, V_j) \coloneqq \min\set{w \in V_j \mid \set{v, w} \in E} $
denote the smallest element in~$ V_j $
which is compatible to a given $ v \in V \setminus V_j $.
\begin{theorem}(\cite[Proposition~3.2 + Theorem~3.4]{staircase2018})
	Let $ P(G, \VV) $ be the CPMC polytope as introduced in \cref{Thm:CPMCF}.
	If the partition $ \VV $ is staircase,
	then $ P(G, \VV) $ is completely described by the constraints
	\begin{alignat}{1}
		\sum_{v \in V_i} x_v &= 1 \quad \forall V_i \in \VV, \label{Eq:MC}\\
		\sum_{\substack{u \in V_i:\\u \geq_i v}} x_u
			&\leq \sum_{\substack{w \in V_j:\\w \geq_j \min(v, V_j)}} x_w
			\quad \forall V_i \in \VV, \forall v \in V_i, \forall V_j \in \VV, j \neq i,
			\label{Eq:Comp}\\
		x &\geq 0. \label{Eq:NN}
	\end{alignat}
	Moreover, the constraint matrix of system~\cref{Eq:MC,Eq:Comp,Eq:NN}
	is totally unimodular.
\end{theorem}
\begin{proof}
	First, observe that any $ x \in P(G, \VV) $
	fulfils constraints~\cref{Eq:MC,Eq:NN} by definition.
	It also fulfils constraint~\cref{Eq:Comp}
	as the left-hand side it produces is never bigger than~$1$
	and choosing some node in a set~$ V_i \in \VV $
	requires the choice of a compatible element in all other $ V_j \in \VV $.
	
	To prove the convex-hull property,
	let $ h \in \R^{\card{V}} $ be a point fulfilling~\cref{Eq:MC,Eq:Comp,Eq:NN}.
	We define a subset $ S_v \subseteq U$ for each node $ v \in V $
	as described in routine \textsc{Define-CPMCS-Subsets} in~\cref{Fig:CPMCS}.
	Observe that the subroutine \textsc{Match} precisely sets these subsets
	such that $ S_v = [\sum_{u \in V_i: u <_i v} x_u,
		\sum_{u \in V_i: u \leq_i v} x_u] $ holds
	for all $ v \in V_i $, $ V_i \in \VV $.
	Now consider an arbitrary subgraph~$ \Gij $ of~$G$.
	Constraint~\cref{Eq:MC} then ensures $ \bigcup_{v \in V_i} S_v
		= \bigcup_{w \in V_j} S_w = 1 $.
	Further, \cref{Eq:Comp} implies $ \bigcup_{u \in V_i:\, u \geq_i v} S_u
		\subseteq \bigcup_{w \in V_j:\, w \geq_j \min(v, V_j)} S_w $
	for all $ v \in V_i $.
	
	Finally, the total unimodularity is shown in \cite[Theorem~3.4]{staircase2018}.
\end{proof}
\begin{figure}[h]
	\null\vfill\null
	\begin{algorithmic}[1]
		\Function{Define-CPMCS-Subsets}{}
		\For{$ V_i \in \VV $}
			\State Let $ (i_1, \ldots, i_p) $ be the ordering
				of the elements in $ V_i $ according to $ <_i $
			\State $ (\S_{i_1}, \ldots, \S_{i_p}) \coloneqq
				\textsc{Match}([0, 1), (h_{i_1}, \ldots, h_{i_p})) $
		\EndFor
		\EndFunction
	\end{algorithmic}
	\caption{Routine \textsc{Define-CPMCS-Sets}}
	\label{Fig:CPMCS}
	\null\vfill\null
\end{figure}

\subsection{The stable-set problem on a bipartite graph}
\label{Sec:Bipartite_Stable_Set}

We briefly recapitulate an example from \cite{zuckerberg2016geometric},
namely a proof that the stable-set problem on a bipartite graph~$ G = (V, E) $
is completely described by the stable-set inequalities.
Our framework of set characterizations allows us
to present it in a more concise form than previously possible.
\begin{theorem}
	Let $ P(G) \coloneqq \set{x \in \F^{\card{E}} \mid
		\cref{stab:set-inq}} $
	be the stable-set polytope,
	and let $ H(G) \coloneqq \set{x \in \FR^{\card{E}} \mid
		\cref{stab:set-inq}} $
	be its stable-set relaxation.
	Then we have $ P = H $.
\end{theorem}
\begin{proof}
	It is obvious that $ P \subseteq H $.
	For the converse, we consider the set characterization of \cref{stab:set-inq}
	given by
	\begin{alignat}{1}
		S_i \cap \S_j = \emptyset &\quad \forall (i, j) \in E.\label{stabelset:char1}
	\end{alignat}
	Further, let $U$ and $W$ be the two bipartite subsets of~$V$.
	For each point $ h \in H $,
	we then need to find sets $ S_e $ for all $ e \in E $ for each point $ h \in H $
	such that $ \mu(S_e) = h_e $ and the above condition holds.
	These sets are defined in routine \textsc{Define-Bipartite-Stable-Sets-Subsets}
	given in \cref{fig:stable-set}.
	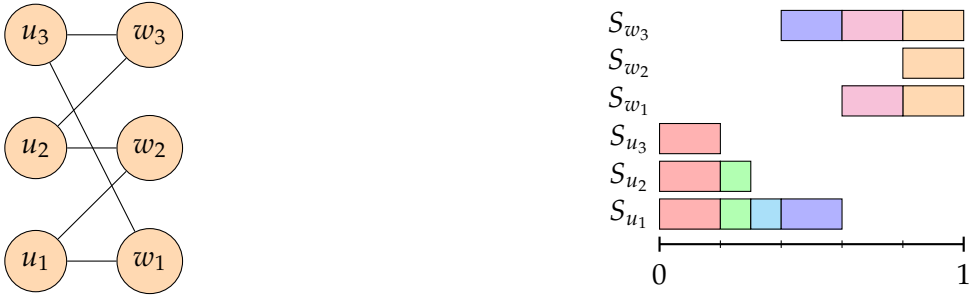
\begin{figure}
		\centering
		\begin{subfigure}{\textwidth}
			\begin{algorithmic}[1]
				\Function{Define-Bipartite-Stable-Set-Subsets}{}
					\For{each $ u \in U $}
						\State $ S_u \coloneqq [0, h_u) $
					\EndFor
					\For{each $ w \in W $}
						\State $ S_w \coloneqq [1 - h_w, 1) $
					\EndFor
				\EndFunction
			\end{algorithmic}
		\end{subfigure}
		\par\vskip\floatsep
		\begin{subfigure}{.5\textwidth}
			\hspace{1cm}
			\begin{tikzpicture}[xscale=1.5,yscale=1.5]
				\node[shape=circle,draw=black, fill=orange!30] (A) at (0,0) {$u_1$};
				\node[shape=circle,draw=black, fill=orange!30] (B) at (0,1) {$u_2$};
				\node[shape=circle,draw=black, fill=orange!30] (C) at (0,2) {$u_3$};
				\node[shape=circle,draw=black, fill=orange!30] (D) at (1,0) {$w_1$};
				\node[shape=circle,draw=black, fill=orange!30] (E) at (1,1) {$w_2$};
				\node[shape=circle,draw=black, fill=orange!30] (F) at (1,2) {$w_3$};
				\path [-] (A) edge node[above] {} (D);
				\path [-] (B) edge node[above] {} (E);
				\path [-] (C) edge node[above] {} (F);
				\path [-] (A) edge node[above] {} (E);
				\path [-] (B) edge node[above] {} (F);
				\path [-] (C) edge node[above] {} (D);
			\end{tikzpicture}
		\end{subfigure}
		\hfil
		\begin{subfigure}{.45\textwidth}
			\centering
			\hspace{-1cm}
			\begin{tikzpicture}[xscale=4,yscale=1]
				\draw[thick] (0,0) -- (1,0);
				\draw[thick] (0,.1) -- (0,.-.1);
				\draw (.2,.05) -- (.2,-.05);
				\draw (.4,.05) -- (.4,-.05);
				\draw (.6,.05) -- (.6,-.05);
				\draw (.8,.05) -- (.8,-.05);
				\draw[thick] (1,.1) -- (1,.-.1);
				\node at (0,-.4) {$0$};
				\node at (1,-.4) {$1$};
				
				\newcommand\AONE{.5*0}
				\newcommand\ATWO{.5*1}
				\newcommand\ATHREE{.5*2}
				\newcommand\AFOUR{.5*3}
				\newcommand\AFIVE{.5*4}
				\newcommand\ASIX{.5*5}
				
				\newcommand\BLOCKHEIGHT{0.4}
				\draw[fill=red!30] (0,0.2+\AONE) rectangle ++(0.2,\BLOCKHEIGHT);
				\draw[fill=red!30] (0,0.2+\ATWO) rectangle ++(0.2,\BLOCKHEIGHT);
				\draw[fill=red!30] (0,0.2+\ATHREE) rectangle ++(0.2,\BLOCKHEIGHT);
				
				\draw[fill=green!30] (0.2,0.2+\AONE) rectangle ++(0.1,\BLOCKHEIGHT);
				\draw[fill=green!30] (0.2,0.2+\ATWO) rectangle ++(0.1,\BLOCKHEIGHT);
				
				\draw[fill=cyan!30] (0.3,0.2+\AONE) rectangle ++(0.1,\BLOCKHEIGHT);		
				\draw[fill=blue!30] (0.4,0.2+\AONE) rectangle ++(0.2,\BLOCKHEIGHT);				
				\draw[fill=blue!30] (0.4,0.2+\ASIX) rectangle ++(0.2,\BLOCKHEIGHT);

				\draw[fill=magenta!30] (0.6,0.2+\ASIX) rectangle ++(0.2,\BLOCKHEIGHT);
				\draw[fill=orange!30] (0.8,0.2+\ASIX) rectangle ++(0.2,\BLOCKHEIGHT);				
				
				\draw[fill=orange!30] (0.8,0.2+\AFIVE) rectangle ++(0.2,\BLOCKHEIGHT);
				
				\draw[fill=orange!30] (0.8,0.2+\AFOUR) rectangle ++(0.2,\BLOCKHEIGHT);
				\draw[fill=magenta!30] (0.6,0.2+\AFOUR) rectangle ++(0.2,\BLOCKHEIGHT);

				\node at (-.1,.4+\AONE) {$S_{u_1}$};
				\node at (-.1,.4+\ATWO) {$S_{u_2}$};
				\node at (-.1,.4+\ATHREE) {$S_{u_3}$};
				\node at (-.1,.4+\AFOUR) {$S_{w_1}$};
				\node at (-.1,.4+\AFIVE) {$S_{w_2}$};
				\node at (-.1,.4+\ASIX) {$S_{w_3}$};
			\end{tikzpicture}
		\end{subfigure}
		\caption{Routine \textsc{Define-Bipartite-Stable-Set-Subsets} (top),
			exemplary graph with~$6$ nodes (bottom left)
			and possible output of the routine
			for the point given by $ h_u = (0.6, 0.3, 0.2) $
			and $ h_w = (0.4, 0.2, 0.6) $.
			The point~$h$ can be written as a convex combination of six stable sets,
			namely $ \set{u_1, u_2, u_3} $, $ \set{u_1, u_2} $, $ \set{u_1} $,
			$ \set{u_1, w_3} $, $ \set{w_1, w_3} $ and $ \set{w_1, w_2, w_3} $,
			each one marked with same colour (bottom right).}
		\label{fig:stable-set}
		\null\vfill\null
	\end{figure}
	It is apparent that the so-defined sets satisfy $ \mu(S_e) = h_e $ for all $ e \in E $.
	Due to~\cref{stab:set-inq}, they do not overlap for any edge $ e \in E $
	and thus satisfy~\cref{stabelset:char1}.
	Altogether, we have proved $ H \subseteq P $.
\end{proof}

\subsection{Piecewise linear functions}
\label{Sec:Piecewise}

We will now use our convex extension of Zuckerberg's method
to reprove the integrality of two polytopes
modelling the convex hull a graph of a one-dimensional piecewise linear function.
This is another example which shows that the method is well suited
to prove results for the convex-hull of a mixed-integer set in many cases.
In our overview over the two models, we follow \cite{sridhar2013locally}.

We consider a continuous, piecewise linear function
\begin{equation*}
	\func{f}{[B_0, B_n]}{\R},\quad f(x) = m_i x + b_i,
		\text{ if } x \in [B_{i - 1}, B_i] \quad \forall i \in \set{1, \ldots, n},
\end{equation*}
where we assume $ B_0 = 0 $, $ B_n > 0 $, $ f(0) = 0 $ and $ B_0 < \ldots < B_n $
for the breakpoints of~$f$.
Let $ F_i \coloneqq f(B_i) $ for all $ i \in \set{0, \ldots, n} $.
We are now interested in describing the polytope
$ P \coloneqq \conv \set{(B_i, F_i) \mid i \in \set{0, \ldots, n}} $.
In the literature, there are many well-known
mixed-integer-programming (MIP) formulations for~$P$
(see \cite{sridhar2013locally} for an overview).
An MIP formulation is called \emph{locally ideal}
if the vertices of the linear programming (LP) relaxation
satisfies all integrality requirements and SOS2 properties.
For the multiple-choice method and the incremental method,
we prove via Zuckerberg's method that they are locally ideal.

\subsubsection{Multiple-choice method}

The \emph{multiple-choice method (MCM)}
(also called \emph{lambda} or \emph{SOS2 method})
introduces variables $ x \in \R $ and $ y \in \R $
for the $x$- and $y$-coordinates of the graph of~$f$ respectively.
Further, we need an auxiliary variable $ z \in \F $
and a set of variables $ \lambda \in \R^{n + 1}_+ $
which satisfies the SOS2 property.
This means that at most two of the $ \lambda $-variables can be positive,
and if two of them are positive, the two must be adjacent in the order of the vector.
The model then introduces the following constraints:
\begin{alignat}{1}
	x &= \sum_{i = 1}^n B_i \lambda_i,\label{equ:sos-def1}\\
	y &= \sum_{i = 1}^n F_i \lambda_i,\label{equ:sos-def2}\\
	z &= \sum_{i = 1}^n \lambda_i.\label{equ:sos-def3}
\end{alignat}
The MCM polytope is then given by
\begin{equation*}
	P \coloneqq \conv \SSet{(x, y, z, \lambda)
		\in \R \times \R \times \F \times \R^{n + 1}_+ }
			{\begin{split}
				\text{\cref{equ:sos-def1,equ:sos-def2,equ:sos-def3}},\\
				\lambda \text{ is SOS2}.
			\end{split}}.
\end{equation*}
Its linear relaxation is
\begin{equation*}
	H \coloneqq \SSet{(x, y, z, \lambda)
		\in \R \times \R \times \FR \times \R^{n + 1}_+}
			{\begin{split}
				\text{\cref{equ:sos-def1,equ:sos-def2,equ:sos-def3}},\\
				\lambda_i \geq 0\quad \forall i \in [n]
			\end{split}}.
\end{equation*}
We now show that the two coincide.
\begin{theorem}
	We have $ H = P $. Further, $H$ is integral, which means that MCM is locally ideal.
\end{theorem}
\begin{proof}
	We first translate the constraints defining~$P$ into the set characterizations
	\begin{alignat}{1}
		\phiE{U}(S_x, t) &= \sum_{i = 1}^n B_i \phiE{U}(S_{\lambda_i}, t)
			\quad \forall t \in U,\label{equ:set-char-sos-def1}\\
		\phiE{U}(S_y, t) &= \sum_{i = 1}^n F_i \phiE{U}(S_{\lambda_i}, t)
			\quad \forall t \in U,\label{equ:set-char-sos-def2}\\
		\phiE{U}(S_z, t) &= \sum_{i = 1}^n \phiE{U}(S_{\lambda_i}, t)
			\quad \forall t \in U,\label{equ:set-char-sos-def3}\\
		\phiE{U}(S_z, t) &\leq 1
			\quad \forall t \in U,\label{equ:set-char-sos-def4}\\
		\phiE{U}(S_{\lambda_i}, t) &\geq 0,
			\quad \forall i \in [n], \forall t \in U.\label{equ:set-char-sos-def5}
	\end{alignat}
	For a given point $ h = (h_x, h_y, h_z, h_\lambda) \in H $,
	the set construction is given in routine \textsc{Define-MCM-Subsets}
	in \cref{Fig:sos}.
	\begin{figure}[h]
		\begin{algorithmic}[1]
			\Function{Define-SOS2-Subsets}{}
				\State $ S_x \leftarrow \emptyset $
				\State $ S_y \leftarrow \emptyset $
				\For{$ i \in [n] $}
					\State $ I \leftarrow [\sum_{j = 1}^{i - 1} h_{\lambda_j},
						\sum_{j = 1}^i h_{\lambda_j}) $
					\State $ S_{\lambda_i} \leftarrow (I, 1) $
					\State $ S_x \leftarrow S_x \cup (I, B_i) $
					\State $ S_y \leftarrow S_y \cup (I, F_i) $
				\EndFor
				\State $ S_z \leftarrow ([0, h_z), 1) $
				\State \Return{$ S_x, S_y, S_z $}
			\EndFunction
		\end{algorithmic}
		\centering
		\caption{Routine \textsc{Define-SOS2-Sets}}
		\label{Fig:sos}
	\end{figure}
	The routine first initializes~$ S_{h_x} $ and~$ S_{h_y} $ as empty sets.
	It then iterates over all segments $ i \in [n] $.
	Lines~5 and~6 place the sets for the $ \lambda $-variables next to each other.
	Together with Line~10 this ensures
	that \cref{equ:set-char-sos-def3,equ:set-char-sos-def4,equ:set-char-sos-def5}
	are fulfilled.
	In Lines~7 and~8, we iteratively construct the sets for~$ h_x $
	by adding rectangles of width~$I$ and height~$ B_i $.
	Similarly, we define the sets for~$ h_y $
	by adding rectangles of width~$I$ and height~$ F_i $.
	This ensures that \cref{equ:set-char-sos-def1,equ:set-char-sos-def2} hold.
	It can easily be checked that the sets have the required measures,
	\ie we have $ \muE{U}(S_x) = h_x $, $ \muE{U}(S_y) = h_y $,
	$ \muE{U}(S_z) = h_z $ and $ \muE{U}(S_{\lambda_i}) = h_{\lambda_i} $
	for all $ i \in [n] $.
	As $ \conv(P) \subseteq H $ is obvious,
	we can invoke \cref{thm:extended_zucker}
	and thus conclude $ H = P $.
	
	For each $ t \in U $, there exists at most one $ i \in [n] $
	such that $ y(S_{\lambda_i}, t) = 1 $ holds.
	This shows that the vertex associated with~$t$ via the mapping $ \varphi $
	has at most one $ \lambda $-coordinate set to~$1$
	and the others to~$0$ or all coordinates set to~$0$
	and the $z$-coordinate is at either~$0$ or~$1$.
	We can conclude that all vertices of~$H$
	satisfy the integrality and SOS2 requirements
	in the definition of~$P$.
\end{proof}

\subsubsection{Incremental method}

The \emph{incremental method (IM)} (or \emph{delta method})
also has variables $ x \in \R $ and $ y \in \R $
for the $x$- and $y$-coordinates of the graph of~$f$.
Further, it introduces an auxiliary variable $ z \in \F $
and two ordered sets of variables $ \delta \in \R^n $ and $ b \in \F^{n - 1} $.
They are coupled via the following constraints:
\begin{alignat}{1}
	x &= \sum_{i = 1}^n (B_i - B_{i - 1}) \delta_i\label{equ:inc-def1}\\
	y &= \sum_{i = 1}^n (F_i - F_{i - 1}) \delta_i\label{equ:inc-def2}\\
	\delta_1 &\leq z\label{equ:inc-def3}\\
	0 &\leq \delta_n\label{equ:inc-def4}\\
	\delta_{i + 1} &\leq b_i \leq \delta_i \quad \forall i \in [n - 1].\label{equ:inc-def5}
\end{alignat}
The IM polytope is then given given by
\begin{equation*}
	P \coloneqq \conv\SSet{(x, y, z, \delta, b)
		\in \R \times \R \times \F \times \R^n \times \F^{n + 1}}
			{\cref{equ:inc-def1,equ:inc-def2,equ:inc-def3,equ:inc-def4,equ:inc-def5}},
\end{equation*}
and its linear relaxation is
\begin{equation*}
	H \coloneqq \SSet{(x, y, z, \delta, b)
		\in \R \times \R \times \FR \times \R^n \times \FR^{n + 1}}
			{\cref{equ:inc-def1,equ:inc-def2,equ:inc-def3,equ:inc-def4,equ:inc-def5}}.
\end{equation*}
\begin{theorem}
	We have $ H = P $.
	Further, $H$ is integral,
	which means that IM is locally ideal.
\end{theorem}
\begin{proof}
	The set characterizations corresponding to the constraints of~$P$ read
	\begin{alignat}{1}
		\phiE{U}(S_x, t) &= \sum_{i = 1}^n (B_i - B_{i - 1}) \phiE{U} (S_{\delta_i}, t)
			\quad \forall t \in U,\label{equ:set-char-inc-def1}\\
		\phiE{U}(S_y, t) &= \sum_{i = 1}^n (F_i - F_{i - 1}) \phiE{U}(S_{\delta_i}, t)
			\quad \forall t \in U,\label{equ:set-char-inc-def2}\\
		\phiE{U}(S_{\delta_1}, t) &\leq \phiE{U}(S_z, t)
			\quad \forall t \in U,\label{equ:set-char-inc-def3}\\
		0 &\leq \phiE{U} (S_{\delta_n}, t)
			\quad \forall t \in U,\label{equ:set-char-inc-def4}\\
		\phiE{U}(S_{\delta_{i + 1}}, t)
			&\leq \phiE{U}(S_{b_i}, t)
			\leq \phiE{U}(S_{\delta_i}, t) 
			\quad \forall i \in [n - 1], t \in U.\label{equ:set-char-inc-def5}
	\end{alignat}
	For a given point $ h = (h_x, h_y, h_z, h_\lambda) \in H $
	and $ h_{\delta_{n + 1}} \coloneqq 0 $,
	the set construction is given in routine \textsc{Define-Incremental-Sets}
	in \cref{Fig:incremental}.
	\begin{figure}
		\begin{algorithmic}[1]
			\Function{Define-Incremental-Subsets}{}
				\State $ S_x \leftarrow \emptyset $
				\State $ S_y \leftarrow \emptyset $
				\For{$ i \in [n] $}
					\State $ S_{\delta_i} \leftarrow ([0, h_{\delta_i}), 1) $
					\State $ S_x \leftarrow S_x \cup ([h_{\delta_{i + 1}}, h_{\delta_i}),
						\sum_{j = 1}^i (B_j - B_{j - 1}) h_{\delta_j}) $
					\State $ S_y \leftarrow S_y \cup([h_{\delta_{i + 1}}, h_{\delta_i}),
						\sum_{j = 1}^i (F_j - F_{j - 1}) h_{\delta_j}) $
				\EndFor
				\State $ S_z \leftarrow ([0, h_z), 1) $
				\For{$ i \in [n - 1] $}
					\State $ S_{b_i} \leftarrow ([0, h_{b_i}), 1) $
				\EndFor
				\State \Return{$ S_x, S_y, S_z, S_{b_1}, \ldots, S_{b_{n - 1}} $}
			\EndFunction
		\end{algorithmic}
		\centering
		\caption{Routine \textsc{Define-Incremental-Subsets}}
		\label{Fig:incremental}
	\end{figure}
	We first initialize~$ S_x $ and~$ S_y $ as empty sets.
	Then we iterate over all segments $ i \in [n] $
	and place for each of them the rectangle $ [0, h_{\delta_i}) $
	with a height of~$1$ in Line~5.
	This ensures~\cref{equ:set-char-inc-def4}.
	From~\cref{equ:inc-def5} it follows that $ h_{\delta_{i + 1}} \leq h_{\delta_i} $
	holds for all $ i \in [n - 1] $.
	Together with the steps in Lines~10--12
	we ensure that~\cref{equ:set-char-inc-def5} is satisfied
	and the width of the rectangle defined in Lines~6--7 is non-negative.
	Its height is chosen such that \cref{equ:set-char-inc-def1,equ:set-char-inc-def2}
	are satisfied.
	In Line~9, $ S_z $ is set in such a way that \cref{equ:set-char-inc-def3} holds.
	
	It can easily be verified that the sets have the required measures,
	\ie $ \muE{U}(S_x) = h_x $, $ \muE{U}(S_y) = h_y $, $ \muE{U}(S_z) = h_z $,
	$ \muE{U}(S_{\delta_i}) = h_{\delta_i} $ for all $ i \in [n] $
	and $ \muE{U}(S_{b_i}) = h_{b_i} $ for all $ i \in [n - 1] $.
	Further, $ P \subseteq H_{\text{INC}} $ is obvious.
	Thus, we can use \cref{thm:extended_zucker}
	to conclude $ P = H $.
	
	For each $ t \in U $, there exists at most one $ i \in [n] $
	such that $ y(S_{\lambda_i}, t) = 1 $.
	Therefore, the vertex associated with~$t$ via the mapping $ \varphiE{U} $
	has all $b$-coordinates equal to either~$0$ or~$1$,
	and the same holds for the $z$-coordinate.
	This means that all vertices satisfy the integrality requirements
	in the definition of~$P$.
\end{proof}

\subsection{Total unimodularity of interval matrices}
\label{Sec:TU_Interval}

A matrix $ A \subseteq \F^{n \times m} $
is called an \emph{interval matrix}
if the $1$-entries in each row are consecutive.
Let~$A$ be such a matrix, and let $ b \in \Z^n $.
For each $ i \in [n] $, we introduce a variable $ x_i \in \R_+ $.
We then consider the system $ Ax \leq b $, which we write constraint-wise:
\begin{equation}
	\sum_{i \in [n]:\, A_{ij} = 1} x_i \leq b_j \quad \forall j \in [m].
	\label{interval:set-inq}
\end{equation}
In this setting, we can reprove the total unimodularity of~$A$
via Hoffman's and Kruskal's theorem.
\begin{theorem}
	Let $ P \coloneqq \conv \set{x \in \N^{\card{E}} \mid \cref{interval:set-inq}} $,
	and let $ H \coloneqq \set{x \in \R_+^{\card{E}} \mid \cref{interval:set-inq}} $.
	Then we have $ H = P $.
	\label{interval:thm}
\end{theorem}
\begin{proof}
	Clearly, we have $ P \subseteq H $.
	In order to prove the converse,
	we transform constraint~\cref{interval:set-inq}
	into a set characterization, namely
	\begin{alignat}{1}
		\sum_{i \in [n]:\, A_{ij} = 1} \phiE{U}(S_i, t) \leq b_j
			& \quad \forall j \in [m], \forall t \in U.
		\label{interval:char1}
	\end{alignat}
	We assume $ b_j \geq 0 $ for all $ j \in [m] $,
	since otherwise the polytope is empty.
	For each point $ h \in H $, we next need to find sets~$ S_i $ for all $ i \in [n] $
	such that $ \muE{U}(S_i) = h_i $ and the above conditions hold.
	The sets are defined in routine \textsc{Define-interval-matrices-Subsets}
	given in \cref{Fig:interval-extended}.
	\begin{figure}[h]
		\null\vfill\null
		\begin{algorithmic}[1]
			\Function{Define-interval-matrices-Subsets}{}
				\State $ t \leftarrow 0 $
				\For{$ i = 1, \ldots, n $}
					\Comment{Process elements in the order of the columns}
					\State $ (I, t) \leftarrow o(t, h_i - \myfloor{h_i}) $
					\State $ S_i \coloneqq ([0, 1), \myfloor{h_i}) \cup (I, 1) $
				\EndFor
			\EndFunction
		\end{algorithmic}
		\null\vfill\null
		\centering
		\caption{Routine \textsc{Define-interval-matrices-Subsets}}
		\label{Fig:interval-extended}
		\null\vfill\null
	\end{figure}
	From the above construction, it is apparent that for each $ h \in H $
	the corresponding sets satisfy~\cref{interval:char1},
	because they are placed in consecutive order.
	Thus, we have proved $ H \subseteq P $.
\end{proof}
The desired result is now a consequence of~\cref{interval:thm,hoffi:thm}.
\begin{corollary}
	Let $A$ be an interval matrix. Then~$A$ is totally unimodular.
\end{corollary}	

\subsection{Proof of \cref{cor:set_interpretation3}}
\label{sec:proofs-graph}

In order to prove \cref{cor:set_interpretation3},
we first need to characterize the convex hull
of the graph of~$f$.
To this end, we define the two functions $ \psi_-, \psi_+\colon T \to \R $ via
\begin{align*}
	\psi_-(x) &= \min\SSet{\sum_{\xi \in \FF} \mu(L_\xi(S_1, \ldots, S_n)) f(\xi)}
		{(S_1, \ldots, S_n) \in Z(x)},\\
	\psi_+(x) &= \max\SSet{\sum_{\xi \in \FF} \mu(L_\xi(S_1, \ldots, S_n)) f(\xi)}
		{(S_1, \ldots, S_n) \in Z(x)}.
\end{align*}
Recall that we assume $ \FF \subseteq \F^n $ in this case.
\begin{lemma}
	\label{thm:zucker_function_small}
	For every function $ f\colon \FF \to \R $ with $ \FF \subseteq \F^n $,
	we have
	\begin{equation*}
		X(f) = \SSet{(x, z) \in T \times \R}{\psi_-(x) \leq z \leq \psi_+(x)}.
	\end{equation*}  
\end{lemma}
\begin{proof}
	First, assume $ (x, z) \in X(f) $, which implies
	\begin{equation*}
		(x, z) = \sum_{k = 1}^{\card{\FF}} \lambda_k (\xi^k, f(\xi^k))
	\end{equation*}
	for some $ \lambda_k \geq 0 $ for $ k = 1, \ldots, \card{F} $
	with $ \sum_{k = 1}^{\card{F}} \lambda_k = 1 $
	and $ \xi^1, \ldots, \xi^k $ is a fixed ordering of~$ \FF $.
	For the partition $ U = I_1 \cup \ldots \cup I_{\card{\FF}} $
	with $ I_1 = [0, \lambda_1) $ and $ I_k = [\lambda_1 + \ldots + \lambda_{k - 1},
		\lambda_1 + \ldots + \lambda_k) $
	for $ k \in \set{2, \ldots, \card{\FF}} $, we set
	\begin{equation*}
		S_i = \bigcup_{k:\, \xi^k_i \neq 0} I_k.
	\end{equation*}
	For $ k = 1, \ldots, \card{\FF} $,
	we have $ \LL^U_{\xi^k}(S_1, \ldots, S_n) = I_k $,
	and, as a consequence, $ \mu^U(\LL^U_{\xi^k}(S_1, \ldots, S_n)) = \lambda_k $.
	With
	\begin{equation*}
		z = \sum_{k = 1}^{\card{\FF}} \lambda_k \psi(\xi^k)
			= \sum_{k = 1}^{\card{\FF}} \mu^U(\LL^U_{\xi^k}(S_1, \ldots, S_n)) \psi(\xi^k),
	\end{equation*}
	it follows that $ \psi_-(x) \leq z \leq \psi_+(x) $.
	
	For the converse, assume $ \psi_-(x) \leq z \leq \psi_+(x) $
	and let $ (S_1, \ldots, S_n) $ as well as $ (S'_1, \ldots, S'_n) $
	be optimizers for the problems
	defining $ \psi_-(x) $ and $ \psi_+(x) $ respectively.
	We write $ z = t\psi_-(x) + (1 - t) \psi_+(x) $ for some $ t \in [0, 1] $
	and set
	\begin{equation*}
		\lambda(\xi) = t\mu^U(\LL^U_{\xi}(S_1, \ldots, S_n))
			+ (1 - t) \mu^U(\LL^U_{\xi}(S'_1, \ldots, S'_n))
	\end{equation*}
	for all $ \xi \in \FF $.
	This way we obtain the required representation of $ (x, z) $
	as the convex combination
	\begin{equation*}
		(x, z) = \sum_{\xi \in \FF} \lambda(\xi)(\xi,\psi(\xi)).
	\end{equation*}
\end{proof}
We can then state the proof as follows.
\begin{proof}[of \cref{cor:set_interpretation3}]
	We observe
	\begin{align*}
		\sum_{\xi \in \FF} & \mu (L_{\xi}(S_1, \ldots, S_n)) f(\xi)
			= \sum_{\xi \in \FF} \mu(L_{\xi}(S_1, \ldots, S_n))
				\sum_{i = 1}^k a_i \Psi_i(\xi_1, \ldots, \xi_n)\\
			&= \sum_{i = 1}^k a_i (\sum_{\xi \in \FF} \mu
				(L_\xi(S_1, \ldots, S_n) \Psi_i(\xi_1, \ldots, \xi_n))\\
			&= \sum_{i \in [k]} a_i \Omega(S_1, \ldots, S_n, \Psi_i).\\[-2\baselineskip]
	\end{align*}
\end{proof}

\end{document}